\documentclass[psamsfonts]{amsart}
\usepackage{amssymb,amsfonts}
\usepackage{latexsym}
\usepackage{enumerate}
\usepackage{mathrsfs}
\usepackage{url}
\usepackage{algorithmic}
\usepackage{algorithm}
\usepackage{tikz}
\usepackage{parsetree}

\newtheorem{thm}{Theorem}[section]
\newtheorem{cor}[thm]{Corollary}
\newtheorem{prop}[thm]{Proposition}

\newtheorem{quest}[thm]{Question}

\theoremstyle{definition}
\newtheorem{defn}[thm]{Definition}
\newtheorem{defns}[thm]{Definitions}

\theoremstyle{remark}
\newtheorem{rem}[thm]{Remark}

\numberwithin{equation}{section}

\title{Integer Complexity: Representing numbers of bounded defect}
\author{Harry Altman}
\date{March 6, 2016}

\begin{document}

\newcommand{\cpx}[1]{\|#1\|}
\newcommand{\dft}{\delta}
\newcommand{\st}{{st}}
\newcommand{\xpdd}[1]{\hat{#1}}
\newcommand{\drop}{\Delta}
\newcommand{\badfac}{\kappa}
\newcommand{\cR}{R}
\newcommand{\acl}{\ell}
\newcommand{\exnum}{1.92}

\newcommand{\NUM}{{48}}
\newcommand{\TWONUM}{{96}}

\newcommand{\N}{{\mathbb N}}
\newcommand{\R}{{\mathbb R}}
\newcommand{\Z}{{\mathbb Z}}
\newcommand{\Q}{{\mathbb Q}}
\newcommand{\sS}{{\mathcal S}}
\newcommand{\sT}{{\mathcal T}}

\newcommand{\floor}[1]{{\lfloor #1 \rfloor}}
\newcommand{\ceil}[1]{{\lceil #1 \ceil}}

\begin{abstract}
Define $\cpx{n}$ to be the \emph{complexity} of $n$, the smallest number of ones
needed to write $n$ using an arbitrary combination of addition and
multiplication.  John Selfridge showed that $\cpx{n}\ge 3\log_3 n$ for all $n$.
Based on this, this author and Zelinsky defined \cite{paper1} the ``defect'' of
$n$, $\dft(n):=\cpx{n}-3\log_3 n$, and this author showed that the set of all
defects is a well-ordered subset of the real numbers \cite{paperwo}.  This was
accomplished by showing that for a fixed real number $r$, there is a finite set
$S$ of polynomials called ``low-defect polynomials'' such that for any $n$ with
$\dft(n)<r$, $n$ has the form $f(3^{k_1},\ldots,3^{k_r})3^{k_{r+1}}$ for some
$f\in S$.  However, using the polynomials produced by this method, many
extraneous $n$ with $\dft(n)\ge r$ would also be represented.  In this paper we
show how to remedy this and modify $S$ so as to represent precisely the $n$ with
$\dft(n)<r$ and remove anything extraneous.  Since the same polynomial can
represent both $n$ with $\dft(n)<r$ and $n$ with $\dft(n)\ge r$, this is not a
matter of simply excising the appropriate polynomials, but requires
``truncating'' the polynomials to form new ones.
\end{abstract}

\maketitle

\section{Introduction}
\label{intro}

The \emph{complexity} of a natural number $n$ is the least number of $1$'s
needed to write it using any combination of addition and multiplication, with
the order of the operations specified using  parentheses grouped in any legal
nesting.  For instance, $n=11$ has a complexity of $8$, since it can be written
using $8$ ones as
\[ 11=(1+1+1)(1+1+1)+1+1, \] but not with any fewer than $8$.  This notion was
implicitly introduced in 1953 by Kurt Mahler and Jan Popken \cite{MP}; they
actually considered an inverse function, the size of the largest number
representable using $k$ copies of the number $1$.  (More generally, they
considered the same question for representations using $k$ copies of a positive
real number $x$.) Integer complexity was explicitly studied by John Selfridge,
and was later popularized by Richard Guy \cite{Guy, UPINT}.  Following J. Arias
de Reyna \cite{Arias} we will denote the complexity of $n$ by $\cpx{n}$.

Integer complexity is approximately logarithmic; it satisfies the bounds
\begin{equation*}\label{eq1}
3 \log_3 n= \frac{3}{\log 3} \log  n\le \cpx{n} \le \frac{3}{\log 2} \log n  ,\qquad n>1.
\end{equation*}
The lower bound can be deduced from the result of Mahler and Popken, and was
explicitly proved by John Selfridge \cite{Guy}. It is attained with equality for
$n=3^k$ for all $k \ge1$.  The upper bound can be obtained by writing $n$ in
binary and finding a representation using Horner's algorithm. It is not sharp,
and the constant $\frac{3}{\log2} $ can be improved for large $n$ \cite{upbds}.

Based on the above, this author and Zelinsky defined the \emph{defect} of $n$:

\begin{defn}
The \emph{defect} of $n$, denoted $\dft(n)$ is defined by
\[ \dft(n) := \cpx{n} - 3\log_3 n. \]
\end{defn}

The defect has proven to be a useful tool in the study of integer complexity.
For instance, one outstanding question regarding integer complexity, raised by
Guy \cite{Guy}, is that of the complexity of $3$-smooth numbers; is $\cpx{2^n
3^k}$ always equal to $2n+3k$, whenever $n$ and $k$ are not both zero?  This
author and Zelinsky used the study of the defect to show in \cite{paper1} that
this holds true whenever $n\le 21$.

This was accomplished by means of a method for,
given a real number $s$, determining restrictions on what natural numbers $n$
could have $\dft(n)<s$.  They defined:

\begin{defn}
For a real number $s\ge0$, the set $A_s$ is the set of all natural numbers with
defect less than $s$.
\end{defn}

The method worked by first choosing a ``step size'' $\alpha\in(0,1)$, and then
recursively building up coverings for the sets $A_\alpha, A_{2\alpha},
A_{3\alpha}, \ldots$; obviously, any $A_s$ can be reached this way.
Using this method, one can show:

\begin{thm}[Covering theorem]
\label{oldmainthmfront}
For any real $s\ge 0$, there exists a finite set $\sS_s$ of multilinear
polynomials such that for any natural number $n$ satisfying $\dft(n)<s$,
there is some $f\in\sS_s$ and some nonnegative integers
$k_1,\ldots,k_{r+1}$
such that $n=f(3^{k_1},\ldots,3^{k_r})3^{k_{r+1}}$.  In other words, given $s$
one can find $\sS_s$ such that
\[ \{ n : \dft(n)<s \} \subseteq \bigcup_{f\in\sS_s}
\{f(3^{k_1},\ldots,3^{k_r})3^{k_{r+1}} : k_1,\ldots,k_{r+1}\ge 0 \}. \]
\end{thm}

It actually proved more: in particular, the polynomials in
Theorem~\ref{oldmainthmfront} are not arbitrary multilinear polynomials, but
are of a specific form, for which \cite{paperwo} introduced the term
\emph{low-defect polynomials}.  In particular, low-defect polynomials are in
fact \emph{read-once polynomials}, as considered in \cite{ROF} for instance.
See Sections~\ref{review} and \ref{structure} for more on these polynomials.

This sort of theorem is more powerful than it may appear; for instance, one can
use it to show that the defect has unusual order-theoretic properties
\cite{paperwo}:

\begin{thm}(Defect well-ordering theorem)
\label{wothm}
The set $\{ \dft(n) : n \in \mathbb{N} \}$, considered as a subset of the real
numbers, is well-ordered and has order type $\omega^\omega$.
\end{thm}

But while this theorem gave a way of representing a covering of $A_s$, this
covering could include extraneous numbers not actually in $A_s$.  In this paper
we remedy this deficiency, and show that the sets $A_s$ themselves can be
described by low-defect polynomials, rather than low-defect polynomials
merely describing a covering for each $A_s$.

In order to establish this result, we introduce a way of ``truncating'' a
low-defect polynomial $f$ to a given defect $s$, though this replaces one
polynomial $f$ by a finite set of low-defect polynomials $\{g_1,\ldots,g_k\}$.
If we truncate every polynomial in the set $\sS_s$ to the defect $s$, we obtain
a set $\sT_s$ of low-defect polynomials so that for any natural number $n$,
$\dft(n)<s$ if and only if $n=f(3^{k_1},\ldots,3^{k_r})3^{k_{r+1}}$ for some
$f\in \sT_s$ and some $k_1,\ldots,k_{r+1}$.  So as stated above we are no longer
merely covering the set $A_r$, but representing it exactly.
Our main result is as follows. 

\begin{thm}[Representation theorem]
\label{mainthmfront}
For any real $s\ge 0$, there exists a finite set $\sT_s$ of multilinear
polynomials such that a natural number $n$ satisfies $\dft(n)<s$ if and only if
there is some $f\in\sT_s$ and some nonnegative integers $k_1,\ldots,k_{r+1}$
such that $n=f(3^{k_1},\ldots,3^{k_r})3^{k_{r+1}}$.
In other words, given $s$
one can find $\sT_s$ such that
\[ \{ n : \dft(n)<s \} = \bigcup_{f\in\sT_s}
\{f(3^{k_1},\ldots,3^{k_r})3^{k_{r+1}} : k_1,\ldots,k_{r+1}\ge 0 \}. \]
\end{thm}

This theorem is a special case of a stronger result; see Theorem~\ref{mainthm}.

Note that it is possible that, for a given $s$, there will be more than one set
$\sT_s$ satisfying the conclusions of Theorem~\ref{mainthmfront}.  In
particular, it's not clear if the $\sT_s$ generated by the methods of this paper
will be minimal in size.  We ask:

\begin{quest}
For a given $s$, what is the smallest size $g(s)$ of a set $\sT_s$ as above?
\end{quest}

We can also ask what can be said about the function $g(s)$ as $s$ varies.  It is
not monotonic in $s$; for instance, let us consider what happens as $s$
approaches $1$ from below.  We use the classification of numbers of defect less
than $1$ from \cite{paper1}.  For any $s<1$, there's a finite set of numbers $m$
such that any $n$ with $\dft(n)<1$ can be written as $m=n3^k$ for some $k$.  Or
in other words, $\sT_s$ necessarily consists of a finite set of constants.  As
$s$ approaches $1$ from below, the required number of these constants approaches
infinity, i.e., $\lim_{s\to1^-}=\infty$.  However, $g(1)$ is certainly finite,
since all but finitely many of the constants in the $\sS_s$ for $s<1$ can be
grouped together into a single infinite family, $3$-represented by the single
low-defect polynomial $3x+1$.  The $g(s)$ was only required to balloon to
infinity as $s$ approached $1$ from below as for $s<1$, $\sS_s$ could not
contain this short summary, needing to list each possibility separately.

This lack of monotonicity poses an obstacle for attempts to answer this question
simply.  However, one could still possibly obtain a simpler (and potentially
monotonic) answer if one were to restrict the domain of $s$; for instance, if
one required $s$ to be integral.

\subsection{Truncation procedure: An example}

The main new idea of this paper is the truncation procedure. 
We illustrate the truncation procedure by example,
demonstrating it.
For the more general
version, see Section~\ref{sectrunc}.

If we want to describe the set $A_s$, we can first use
Theorem~\ref{oldmainthmfront} to obtain a description of a covering set $\sS$
for $A_s$.  This is the ``building-up'' step.  Then we apply the truncation
procedure to each element of $\sS$; this is the``filtering-down'' step.  As an
example,
we will consider truncating the polynomial \[f(x_1,x_2)=(2x_1+1)x_2+1\]
to the defect value $s=\exnum$.

Observe that for any $k_1, k_2,$ and $k_3$, one has
\[\cpx{f(3^{k_1},3^{k_2})3^{k_3}}\le 4+3k_1+3k_2+3k_3,\] as illustrated by
Figure~\ref{thefig}.  Let us take this as the
``supposed complexity'' of this number.  Then the ``supposed defect'', obtained
by subtracting $3\log_3(f(3^{k_1},3^{k_2})3^{k_3})$ from the ``supposed
complexity'', is equal to
\[4-3\log_3(2+3^{-k_1}+3^{-k_1-k_2}).\]

Now, in reality the actual defect may be less than the ``supposed defect''; but
we will ignore this for now and just work with the ``supposed defect'', which we
know how to compute.  As it will turn out, using the ``supposed defect'' will
still yield the correct result, and we do not need to determine the actual
defect; see remark (2) below.

\begin{figure}[htb]
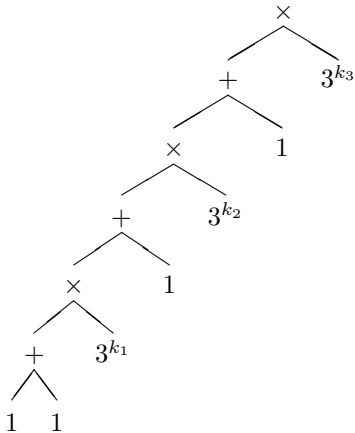

\label{thefig}
\caption{A tree for building the number
$((2\cdot3^{k_1}+1)3^{k_2}+1)3^{k_3}$; note that $3^{k_i}$ has complexity
$3k_i$.}
\begin{parsetree}
( .$\times$.
	( .$+$.
		( .$\times$.
			( .$+$.
				( .$\times$.
					( .$+$.
						`$1$'
						`$1$'
					)
					`$3^{k_1}$'
				)
				`$1$'
			)
			`$3^{k_2}$'
		)
		`$1$'
	)
	`$3^{k_3}$'
)
\end{parsetree}
\end{figure}

So if we were to truncate $f$ to a defect of $3$, there would be nothing to do;
the defect of any numbers coming from $f$ would already be less than $3$, since
they would be at most $4-3\log_3 2 \approx 2.1$.  But here we are truncating to
$\exnum$, and $f$ can indeed yield defects greater than $\exnum$ (for instance,
choose $k_1=k_2=2$, yielding a defect of approximately $1.94$), so it will not
be as simple as that.

In the above expression for the ``supposed defect'', $k_1$ has a larger effect
on the numerical value than $k_2$ does.  We will determine possible values for
$k_1$.  It can easily be checked that for $k_1\in\{0,1\}$, all the defects
produced this way fall below $\exnum$; for $k_1=2$, some are above and some are
below; and for $k_1\ge 3$, all of them are above $\exnum$.  So we will exclude
$k_1\ge 3$ from consideration, and mark down the possible values of $k_1$ as
being $0, 1$, and $2$; then we will substitute these in to $f$ (noting that
$x_i=3^{k_i}$) to yield new polynomials: $3x_2+1$, $7x_2+1$, and $19x_2+1$.  (We
ignore $k_3$ since it has no effect on the ``supposed defect''.)

We now apply the procedure recursively, truncating each of our new polynomials
to $\exnum$.  As noted above, $3x_2+1$ and $7x_2+1$ only produce defects less
than $\exnum$, so we do not need to do anything further to truncate them.  For
$19x_2+1$, or $19\cdot3^{k_2}+1$, we observe that $k_2=0$ and $k_2=1$ yield a
supposed defect below $\exnum$, while $k_2\ge 2$ yields supposed defects above
$\exnum$.  So once again, we substitute in $k=0$ and $k=1$ to $19\cdot3^{k_2}+1$
to yield the constant polynomials $20$ and $58$, along with our polynomials
$3x_2+1$ and $7x_2+1$ from earlier.  At this point we have reached polynomials
of degree $0$, so our final set of polynomials is $\{3x_2+1, 7x_2+1, 20, 58\}$.
If instead the polynomials had more variables, we would have to continue this
recursion further.

{\bf Remarks.}
(1) The general version is more complicated, because in
general there is not necessarily a linear order on which variables have the most
effect on the defect.  Nonetheless, the general idea of fixing values for
variables, progressing from variables of largest effect to variables of smallest
effect, is retained.

(2) Truncating a polynomial $f$ to a defect $r$ replaces it with a finite set of
polynomials that only produce numbers of defect less than $r$.  This leaves the
question of why applying this procedure to an appropriate set of polynomials
should yield all numbers with defect less than the chosen cutoff, rather than
only some of them.  This is where a stronger version of
Theorem~\ref{oldmainthmfront} -- Theorem~\ref{oldmainthm} below, taken from
\cite{paperwo} -- is needed.  This theorem ensures that, for any given $r$, not
only is there a finite set $\sS_r$ of polynomials which represent all numbers
with defect less than $r$ by substituting in powers of $3$; but in fact that it
represents all such numbers ``with the correct complexity''.  Hence, if we
truncate all of them to a defect of $r$, those with defect at least $r$ will be
filtered out simply by the nature of the procedure; and those with defect less
than $r$ will be kept, since each will be represented ``with the correct
complexity'' by some polynomial, and thus kept when that particular polynomial
is truncated.  See Sections~\ref{review} and \ref{sectrunc} for more detail.

\subsection{Comparison to addition chains}

It is worth discussing here some work analogous to this paper in the study of
addition chains.  An \emph{addition chain} for $n$ is defined to be a sequence
$(a_0,a_1,\ldots,a_r)$ such that $a_0=1$, $a_r=n$, and, for any $1\le k\le r$,
there exist $0\le i, j<k$ such that $a_k = a_i + a_j$; the number $r$ is called
the length of the addition chain.  The shortest length among addition chains for
$n$, called the \emph{addition chain length} of $n$, is denoted $\acl(n)$.
Addition chains were introduced in 1894 by H.~Dellac \cite{Dellac} and
reintroduced in 1937 by A.~Scholz \cite{aufgaben}; extensive surveys on the
topic can be found in Knuth \cite[Section 4.6.3]{TAOCP2} and Subbarao
\cite{subreview}.

The notion of addition chain length has obvious similarities to that of integer
complexity; each is a measure of the resources required to build up the number
$n$ starting from $1$.  Both allow the use of addition, but integer complexity
supplements this by allowing the use of multiplication, while addition chain
length supplements this by allowing the reuse of any number at no additional
cost once it has been constructed.  Furthermore, both measures are approximately
logarithmic; the function $\acl(n)$ satisfies
\[ \log_2 n \le \acl(n) \le 2\log_2 n. \]
 
A difference worth noting is that $\acl(n)$ is actually known to be asymptotic
to $\log_2 n$, as was proved by Brauer\cite{Brauer}, but the function $\cpx{n}$
is not known to be asymptotic to $3\log_3 n$; the value of the quantity
$\limsup_{n\to\infty} \frac{\cpx{n}}{\log n}$ remains unknown.  As mentioned
above, Guy \cite{Guy} has asked whether $\cpx{2^k}=2k$ for $k\ge 1$; if true, it
would make this quantity at least $\frac{2}{\log 2}$.  J.~Iraids
et.~al.~\cite{data2} have checked that this is true for $k\le 39$.

Another difference worth noting is that unlike integer complexity, there is no
known way to compute addition chain length via dynamic programming.
Specifically, to compute integer complexity this way, one may use the fact that
for any $n>1$,

\begin{displaymath}
\cpx{n}=\min_{\substack{a,b<n\in \mathbb{N} \\ a+b=n\ \mathrm{or}\ ab=n}}
	\cpx{a}+\cpx{b}.
\end{displaymath}

By contrast, addition chain length seems to be harder to compute.  Suppose we
have a shortest addition chain $(a_0,\ldots,a_{r-1},a_r)$ for $n$; one might
hope that $(a_0,\ldots,a_{r-1})$ is a shortest addition chain for $a_{r-1}$, but
this need not be the case.  An example is provided by the addition chain
$(1,2,3,4,7)$; this is a shortest addition chain for $7$, but $(1,2,3,4)$ is not
a shortest addition chain for $4$, as $(1,2,4)$ is shorter.   Moreover, there is
no way to assign to each natural number $n$ a shortest addition chain
$(a_0,\ldots,a_r)$ for $n$ such that $(a_0,\ldots,a_{r-1})$ is the addition
chain assigned to $a_{r-1}$ \cite{TAOCP2}. This can be an obstacle both to
computing addition chain length and proving statements about addition chains.

Nevertheless, there are important similarities between integer complexity and
addition chains.  As mentioned above, the set of all integer complexity defects
is a well-ordered subset of the real numbers, with order type $\omega^\omega$.
We might also define the notion of \emph{addition chain defect}, defined by
\[\dft^{\acl}(n):=\acl(n)-\log_2 n;\]
for as shown \cite{adcwo} by this author, Theorem~\ref{wothm} has an analogue
for addition chains:

\begin{thm}[Addition chain well-ordering theorem]
The set $\{ \dft^{\acl}(n) : n \in \mathbb{N} \}$, considered as a subset of
the real numbers, is well-ordered and has order type $\omega^\omega$.
\end{thm}

Theorem~\ref{mainthmfront} seems to have a partial analogue for addition chains
in the work of A.~Flammenkamp \cite{fk}.  As mentioned above, this author
introduced the addition chain defect $\dft^{\acl}(n)$, but a closely related
quantity, the number of \emph{small steps} of $n$, was introduced by Knuth
\cite{TAOCP2}.  The number of small steps of $n$ is defined by \[ s(n) :=
\acl(n) - \lfloor \log_2 n \rfloor;\] clearly, this is related to
$\dft^{\acl}(n)$ by $s(n)=\lceil\dft^{\acl}(n)\rceil$.

In 1991, A.~Flammenkamp \cite{fk} determined a method for producing descriptions of all
numbers $n$ with $s(n)\le k$ for a given integer $k$, and produced such
descriptions for $k\le 3$.  Note that for $k$ an integer, $s(n)\le k$
if and only if $\dft^{\acl}(n)\le k$, so this is the same as determining all $n$
with $\dft^{\acl}(n)\le k$, restricted to the case where $k$ is an integer.
Part of what Flammenkamp proved may be summarized as the following:

\begin{thm}[Flammenkamp]
For any integer $k\ge 0$, there exists a finite set $\sS_k$ of polynomials (in
any number of variables, with nonnegative integer coefficients) such that for
any $n$, one has $s(n)\le k$ if and only if one can write
$n=f(2^{m_1},\ldots,2^{m_r})2^{m_{r+1}}$ for some $f\in \sS_k$ and some integers
$m_1,\ldots,m_{r+1}\ge0$.
\end{thm}

Note that this only allows integer $k$, as opposed to
Theorem~\ref{mainthmfront}, which allows arbitrary real $s$.  Also, the polynomials used in Flammenkamp's method are more complicated
than those produced by Theorem~\ref{mainthmfront}; for instance, they cannot
always be taken to be multilinear.

\section{The defect, stability, and low-defect polynomials}
\label{review}

In this section we review the results of \cite{paper1} and \cite{paperwo}
regarding the defect $\dft(n)$, the stable complexity $\cpx{n}_\st$ and stable
defect $\dft_\st(n)$ described below, and low-defect polynomials.

\subsection{The defect and stability}
\label{subsecdft}

First, some basic facts about the defect:

\begin{thm}
\label{oldprops}
We have:
\begin{enumerate}
\item For all $n$, $\dft(n)\ge 0$.
\item For $k\ge 0$, $\dft(3^k n)\le \dft(n)$, with equality if and only if
$\cpx{3^k n}=3k+\cpx{n}$.  The difference $\dft(n)-\dft(3^k n)$ is a nonnegative
integer.
\item A number $n$ is stable if and only if for any $k\ge 0$, $\dft(3^k
n)=\dft(n)$.
\item If the difference $\dft(n)-\dft(m)$ is rational, then $n=m3^k$ for some
integer $k$ (and so $\dft(n)-\dft(m)\in\mathbb{Z}$).
\item Given any $n$, there exists $k$ such that $3^k n$ is stable.
\item For a given defect $\alpha$, the set $\{m: \dft(m)=\alpha \}$ has either
the form $\{n3^k : 0\le k\le L\}$ for some $n$ and $L$, or the form $\{n3^k :
0\le k\}$ for some $n$.  This latter occurs if and only if $\alpha$ is the
smallest defect among $\dft(3^k n)$ for $k\in \mathbb{Z}$.
\item If $\dft(n)=\dft(m)$, then $\cpx{n}=\cpx{m} \pmod{3}$.
\item $\dft(1)=1$, and for $k\ge 1$, $\dft(3^k)=0$.  No other integers occur as
$\dft(n)$ for any $n$.
\item If $\dft(n)=\dft(m)$ and $n$ is stable, then so is $m$.
\end{enumerate}
\end{thm}

\begin{proof}
Parts (1) through (8), excepting part (3), are just Theorem~2.1 from
\cite{paperwo}.  Part (3) is Proposition~12 from \cite{paper1}, and part (9) is
Proposition~3.1 from \cite{paperwo}.
\end{proof}

Also, although it will not be a focus of this paper, we will sometimes want to
consider the set of all defects:

\begin{defn}
We define the \emph{defect set} $\mathscr{D}$ to be $\{\dft(n):n\in\N\}$, the
set of all defects.
\end{defn}

The paper \cite{paperwo} also defined the notion of a \emph{stable defect}:

\begin{defn}
We define a \emph{stable defect} to be the defect of a stable number.
\end{defn}

Because of part (9) of Theorem~\ref{oldprops}, this definition makes sense; a
stable defect $\alpha$ is not just one that is the defect of some stable number,
but one for which any $n$ with $\dft(n)=\alpha$ is stable.  Stable defects can
also be characterized by the following proposition from \cite{paperwo}:

\begin{prop}
\label{modz1}
A defect $\alpha$ is stable if and only if it is the smallest
$\beta\in\mathscr{D}$ such that $\beta\equiv\alpha\pmod{1}$.
\end{prop}

We can also define the \emph{stable defect} of a given number, which we denote
$\dft_\st(n)$.

\begin{defn}
For a positive integer $n$, define the \emph{stable defect} of $n$, denoted
$\dft_\st(n)$, to be $\dft(3^k n)$ for any $k$ such that $3^k n$ is stable.
(This is well-defined as if $3^k n$ and $3^\ell n$ are stable, then $k\ge \ell$
implies $\dft(3^k n)=\dft(3^\ell n)$, and so does $\ell\ge k$.)
\end{defn}

Note that the statement ``$\alpha$ is a stable defect'', which earlier we were
thinking of as ``$\alpha=\dft(n)$ for some stable $n$'', can also be read as the
equivalent statement ``$\alpha=\dft_\st(n)$ for some $n$''.

We then have the following facts relating the notions of $\cpx{n}$, $\dft(n)$,
$\cpx{n}_\st$, and $\dft_\st(n)$:

\begin{prop}
\label{stoldprops}
We have:
\begin{enumerate}
\item $\dft_\st(n)= \min_{k\ge 0} \dft(3^k n)$
\item $\dft_\st(n)$ is the smallest $\alpha\in\mathscr{D}$ such that
$\alpha\equiv \dft(n) \pmod{1}$.
\item $\cpx{n}_\st = \min_{k\ge 0} (\cpx{3^k n}-3k)$
\item $\dft_\st(n)=\cpx{n}_\st-3\log_3 n$
\item $\dft_\st(n) \le \dft(n)$, with equality if and only if $n$ is stable.
\item $\cpx{n}_\st \le \cpx{n}$, with equality if and only if $n$ is stable.
\end{enumerate}
\end{prop}

\begin{proof}
These are just Propositions~3.5, 3.7, and 3.8 from \cite{paperwo}.
\end{proof}

\subsection{Low-defect polynomials and low-defect pairs}

As has been mentioned in Section~\ref{intro}, we are going to represent
the set $A_r$ by substituting in powers of $3$ into certain multilinear
polynomials we call \emph{low-defect polynomials}.  We will associate with each
one a ``base complexity'' to from a \emph{low-defect pair}.  In this section we
will review the basic properties of these polynomials.  First, their definition:

\begin{defn}
\label{polydef}
We define the set $\mathscr{P}$ of \emph{low-defect pairs} as the smallest
subset of $\Z[x_1,x_2,\ldots]\times \N$ such that:
\begin{enumerate}
\item For any constant polynomial $k\in \N\subseteq\Z[x_1, x_2, \ldots]$ and any
$C\ge \cpx{k}$, we have $(k,C)\in \mathscr{P}$.
\item Given $(f_1,C_1)$ and $(f_2,C_2)$ in $\mathscr{P}$, we have $(f_1\otimes
f_2,C_1+C_2)\in\mathscr{P}$, where, if $f_1$ is in $r_1$ variables and $f_2$ is
in $r_2$ variables,
\[ (f_1\otimes f_2)(x_1,\ldots,x_{r_1+r_2}) :=
	f_1(x_1,\ldots,x_{r_1})f_2(x_{r_1+1},\ldots,x_{r_1+r_2}). \]
\item Given $(f,C)\in\mathscr{P}$, $c\in \N$, and $D\ge \cpx{c}$, we have
$(f\otimes x_1 + c,C+D)\in\mathscr{P}$ where $\otimes$ is as above.
\end{enumerate}

The polynomials obtained this way will be referred to as \emph{low-defect
polynomials}.  If $(f,C)$ is a low-defect pair, $C$ will be called its
\emph{base complexity}.  If $f$ is a low-defect polynomial, we will define its
\emph{absolute base complexity}, denoted $\cpx{f}$, to be the smallest $C$ such
that $(f,C)$ is a low-defect pair.
We will also associate to a low-defect polynomial $f$ the \emph{augmented
low-defect polynomial}
\[ \xpdd{f} = f\otimes x_1 \]
\end{defn}

Note that the degree of a low-defect polynomial is also equal to the number of
variables it uses; see Proposition~\ref{polystruct}.  We will often refer to the
``degree'' of a low-defect pair $(f,C)$; this refers to the degree of $f$.  Also
note that augmented low-defect polynomials are never low-defect polynomials; as
we will see in a moment (Proposition~\ref{polystruct}), low-defect polynomials
always have nonzero constant term, whereas augmented low-defect polynomials
always have zero constant term.  We can also observe, as mentioned above, that
low-defect polynomials are in fact read-once polynomials.

Note that we do not really care about what variables a low-defect polynomial (or
pair) is in -- if we permute the variables of a low-defect polynomial or replace
them with others, we will still regard the result as a low-defect polynomial.
From this perspective, the meaning of $f\otimes g$ could be simply regarded as
``relabel the variables of $f$ and $g$ so that they do not share any, then
multiply $f$ and $g$''.  Helpfully, the $\otimes$ operator is associative not
only with this more abstract way of thinking about it, but also in the concrete
way it was defined above.

In \cite{paperwo} were proved the following propositions about low-defect
pairs:

\begin{prop}
\label{polystruct}
Suppose $f$ is a low-defect polynomial of degree $r$.  Then $f$ is a
polynomial in the variables $x_1,\ldots,x_r$, and it is a multilinear
polynomial, i.e., it has degree $1$ in each of its variables.  The coefficients
are non-negative integers.  The constant term is nonzero, and so is the
coefficient of $x_1\ldots x_r$, which we will call the \emph{leading
coefficient} of $f$.
\end{prop}

\begin{prop}
\label{basicub}
If $(f,C)$ is a low-defect pair of degree $r$, then
\[\cpx{f(3^{n_1},\ldots,3^{n_r})}\le C+3(n_1+\ldots+n_r).\]
and
\[\cpx{\xpdd{f}(3^{n_1},\ldots,3^{n_{r+1}})}\le C+3(n_1+\ldots+n_{r+1}).\]
\end{prop}

\begin{proof}
This is a combination of Proposition~4.5 and Corollary~4.12 from \cite{paperwo}.
\end{proof}

Because of this, it makes sense to define:

\begin{defn}
Given a low-defect pair $(f,C)$ (say of degree $r$) and a number $N$, we will
say that $(f,C)$ \emph{efficiently $3$-represents} $N$ if there exist
nonnegative integers $n_1,\ldots,n_r$ such that
\[N=f(3^{n_1},\ldots,3^{n_r})\ \textrm{and}\ \cpx{N}=C+3(n_1+\ldots+n_r).\]
We will say $(\xpdd{f},C)$ efficiently
$3$-represents $N$ if there exist $n_1,\ldots,n_{r+1}$ such that
\[N=\xpdd{f}(3^{n_1},\ldots,3^{n_{r+1}})\ \textrm{and}\ 
\cpx{N}=C+3(n_1+\ldots+n_{r+1}).\]
More generally, we will also say $f$ $3$-represents $N$ if there exist
nonnegative integers $n_1,\ldots,n_r$ such that $N=f(3^{n_1},\ldots,3^{n_r})$.
and similarly with $\xpdd{f}$.
\end{defn}

Note that if $(f,C)$ (or $(\xpdd{f},C)$) efficiently $3$-represents some $N$,
then $(f,\cpx{f})$ (respectively, $(\xpdd{f},\cpx{f})$ efficiently
$3$-represents $N$, which means that in order for $(f,C)$ (or $(\xpdd{f},C)$ to
$3$-represent anything efficiently at all, we must have $C=\cpx{f}$.  However it
is still worth using low-defect pairs rather than just low-defect polynomials
since we may not always know $\cpx{f}$.  In our applications here, where we want
to compute things, taking the time to compute $\cpx{f}$, rather than just making
do with an upper bound, may not be desirable.

For this reason it makes sense to use ``$f$ efficiently $3$-represents $N$'' to
mean ``some $(f,C)$ efficiently $3$-represents $N$'' or equivalently
``$(f,\cpx{f})$ efficiently $3$-reperesents $N$''.  Similarly with $\xpdd{f}$.

In keeping with the name, numbers $3$-represented by low-defect polynomials, or
their augmented versions, have bounded defect.  Let us make some definitions
first:

\begin{defn}
Given a low-defect pair $(f,C)$, we define $\dft(f,C)$, the defect of $(f,C)$,
to be $C-3\log_3 a$, where $a$ is the leading coefficient of $f$.  When we are
not concerned with keeping track of base complexities, we will use $\dft(f)$ to
mean $\dft(f,\cpx{f})$.
\end{defn}

\begin{defn}
Given a low-defect pair $(f,C)$ of degree $r$, we define
\[\dft_{f,C}(n_1,\ldots,n_r) =
C+3(n_1+\ldots+n_r)-3\log_3 f(3^{n_1},\ldots,3^{n_r}).\]
\end{defn}

Then we have:

\begin{prop}
\label{dftbd}
Let $(f,C)$ be a low-defect pair of degree $r$, and let $n_1,\ldots,n_{r+1}$ be
nonnegative integers.
\begin{enumerate}
\item We have
\[ \dft(\xpdd{f}(3^{n_1},\ldots,3^{n_{r+1}}))\le \dft_{f,C}(n_1,\ldots,n_r)\]
and the difference is an integer.
\item We have \[\dft_{f,C}(n_1,\ldots,n_r)\le\dft(f,C)\]
and if $r\ge 1$, this inequality is strict.
\end{enumerate}
\end{prop}

\begin{proof}
This is a combination of Proposition~4.9 and Corollary~4.14 from \cite{paperwo}.
\end{proof}

In fact, not only is $\dft(f,C)$ an upper bound on the values of $\dft_{f,C}$,
it is the least upper bound:

\begin{prop}
\label{supdfts}
Let $(f,C)$ be a low-defect pair, say of degree $r$.  Then $\dft_{f,C}$ is a
strictly increasing function in each variable, and
\[ \dft(f,C) = \sup_{k_1,\ldots,k_r} \dft_{f,C}(k_1,\ldots,k_r).\]
\end{prop}

\begin{proof}
We can define $g$, the reverse polynomial of $f$:
\[g(x_1,\ldots,x_r)=x_1\ldots x_r f(x_1^{-1},\ldots,x_r^{-1}).\]
So $g$ is a multilinear polynomial in $x_1, \ldots, x_r$, with the coefficient
of $\prod_{i\in S} x_i$ in $g$ being the coefficient of $\prod_{i\notin S} x_i$
in $f$.  By Proposition~\ref{polystruct}, $f$ has nonnegative coefficients, so
so does $g$; since the constant term of $f$ does not vanish, the $x_1\ldots x_r$
term of $g$ does not vanish. Hence $g$ is strictly increasing in each variable.

Then
\begin{eqnarray*}
\dft_{f,C}(k_1,\ldots,k_r)=C+3(k_1+\ldots+k_r)-3\log_3 f(3^{k_1},\ldots,3^{k_r})
\\ = C-3\log_3 \frac{f(3^{k_1},\ldots,3^{k_r})}{3^{k_1+\ldots+k_r}}
= C-3\log_3 g(3^{-k_1},\ldots,3^{-k_r})
\end{eqnarray*}
which is strictly increasing in each variable, as claimed.  Furthermore, if
$a$ is the leading coefficient of $f$, then it is also the constant term of $g$,
and so
\[\inf_{k_1,\ldots,k_r} g(3^{-k_1},\ldots,3^{-k_r})=a.\]
Thus
\[ \sup_{k_1,\ldots,k_r} \dft_{f,C}(k_1,\ldots,k_r)=C-3\log_3 a=\dft(f,C).\]
\end{proof}

With this, we have the basic properties of low-defect polynomials.

\subsection{Describing numbers of small defect}
\label{buildsec}

Now we will briefly discuss the ``building-up'' method from \cite{paper1} and
\cite{paperwo} that restricts what numbers may lie in $A_r$.  The new
``filtering-down'' half, truncation, will have to wait for
Section~\ref{sectrunc}.

First, we will need the idea of a \emph{leader}:

\begin{defn}
A natural number $n$ is called a \emph{leader} if it is the smallest number with
a given defect.  By part (6) of Theorem~\ref{oldprops}, this is equivalent to
saying that either $3\nmid n$, or, if $3\mid n$, then $\dft(n)<\dft(n/3)$, i.e.,
$\cpx{n}<3+\cpx{n/3}$.
\end{defn}

Let us also define:

\begin{defn}
For any real $r\ge0$, define the set of {\em $r$-defect numbers} $A_r$ to be 
\[A_r := \{n\in\mathbb{N}:\dft(n)<r\}.\]
Define the set of {\em $r$-defect leaders} $B_r$ to be 
\[
B_r:= \{n \in A_r :~~n~~\mbox{is a leader}\}.
\]
\end{defn}

These sets are related by the following proposition from \cite{paperwo}:

\begin{prop}
\label{arbr}
For every $n\in A_r$, there exists a unique $m\in B_r$ and $k\ge 0$ such that
$n=3^k m$ and $\dft(n)=\dft(m)$; then $\cpx{n}=\cpx{m}+3k$.
\end{prop}

Because of this, we can focus on describing $B_r$, and derive $A_r$ from it.

The paper \cite{paper1} showed how to inductively build up coverings of the sets
$B_r$.   It provided the base case \cite{paper1} in the form of the following
theorem:

\begin{thm}
\label{finite}
For every $\alpha$ with $0<\alpha<1$, the set of leaders $B_\alpha$ is a finite
set.
\end{thm}

It then showed how to inductively build up coverings for the sets $B_\alpha$,
$B_{2\alpha}$, $B_{3\alpha}$, again subject to the restriction that $\alpha<1$.
By applying this with $\alpha$ arbitrarily close to $1$, the paper
\cite{paperwo} obtained the following result:

\begin{thm}
\label{oldmainthm}
For any real $r\ge 0$, there exists a finite covering set $\sS_r$ for $B_r$.
Furthermore, we can choose $\sS_r$ such that each $(f,C)\in \sS_r$ has degree at
most $\lfloor r \rfloor$.
\end{thm}

In this paper, we will actually not need the bound on the degree of the
polynomials; the truncation operation will always output low-defect pairs with
appropriately bounded degree, regardless of the degrees of the inputs.  This
allows the possibility of building up a covering $\sS_r$ by a different method,
rather than the particular method of \cite{paperwo}.  This includes the
possibility of using the same method, but with a smaller step size $\alpha$.

\section{Low-defect expressions, the nesting ordering, and structure of
low-defect polynomials}
\label{structure}

In this section we will go further into the structure of low-defect polynomials.
In order to do this, we will investigate the expressions that give rise to them.
That is to say, if we have a low-defect polynomial $f$, it was constructed
according to rules (1)--(3) in Definition~\ref{polydef}; each of these rules
though gives a way not just of building up a polynomial, but an expression.  For
instance, we can build up the polynomial $4x+2$ by using rule (1) to make $2$,
then using rule (3) to make $2x+1$, then using rule (2) to make $2(2x+1)=4x+2$.
The polynomial $4x+2$ itself does not remember its history, of course; but
perhaps we want to remember its history -- in which we do not want to consider
the \emph{polynomial} $4x+2$, but rather the \emph{expression} $2(2x+1)$, which
is different from the expression $4x+2$, which has a different history.

Strictly speaking, it is possible to prove many of the theorems about low-defect
polynomials in this and the next section purely by structural induction, using
just the rules (1)--(3) in Definition~\ref{polydef}.  But introducing low-defect
expressions is more enlightening; it makes it clear why, for instance, the
nesting ordering (see Definition~\ref{defnest}) takes the form of a forest.

So, with that, we define:

\begin{defn}
A \emph{low defect expression} is defined to be a an expression in positive
integer constants, $+$, $\cdot$, and some number of variables, constructed
according to the following rules:
\begin{enumerate}
\item Any positive integer constant by itself forms a low-defect expression.
\item Given two low-defect expressions using disjoint sets of variables, their
product is a low-defect expression.  If $E_1$ and $E_2$ are low-defect
expressions, we will use $E_1 \otimes E_2$ to denote the low-defect expression
obtained by first relabeling their variables to disjoint and then multiplying
them.
\item Given a low-defect expression $E$, a positive integer constant $c$, and a
variable $x$ not used in $E$, the expression $E\cdot x+c$ is a low-defect
expression.  (We can write $E\otimes x+c$ if we do not know in advance that $x$
is not used in $E$.)
\end{enumerate}
\end{defn}

And, naturally, we also define:

\begin{defn}
We define an \emph{augmented low-defect expression} to be an expression of the
form $E\cdot x$, where $E$ is a low-defect expression and $x$ is a variable not
appearing in $E$.  If $E$ is a low-defect expression, we also denote the
augmented low-defect expression $E\otimes x$ by $\xpdd{E}$.
\end{defn}

It is clear from the definitions that evaluating a low-defect expression yields
a low-defect polynomial, and that evaluating an augmented low-defect expression
yields an augmented low-defect polynomial.  Note also that low-defect
expressions are read-once expressions, so, as mentioned earlier, low-defect
polynomials are read-once polynomials.

\subsection{Equivalence and the tree representation}

We can helpfully represent a low-defect expression by a rooted tree, with the
vertices and edges both labeled by positive integers.  Note, some information is
lost in this representation -- but, as it happens, nothing we will care about;
it turns out that while knowing some of the history of a low defect polynomial
is helpful, knowing the full expression it originated from is more than is
necessary.  The tree representation is frequently more convenient to work with
than an expression, as it does away with such problems as, for instance, $4$ and
$2\cdot 2$ being separate expressions.  In addition, trees can be treated more
easily combinatorially; in a sequel paper\cite{seqest}, we will take advantage
of this to estimate how many elements of $A_r$ lie below a given bound $x$.  So
we define:

\begin{defn}
Given a low-defect expression $E$, we define a corresponding \emph{low-defect
tree} $T$, which is a rooted tree where both edges and vertices are labeled with
positive integers.  We build this tree as follows:
\begin{enumerate}
\item If $E$ is a constant $n$, $T$ consists of a single vertex labeled with
$n$.
\item If $E=E'\cdot x + c$, with $T'$ the tree for $E'$, $T$ consists of $T'$
with a new root attached to the root of $T'$.  The new root is labeled with a
$1$, and the new edge is labeled with $c$.
\item If $E=E_1 \cdot E_2$, with $T_1$ and $T_2$ the trees for $E_1$ and $E_2$
respectively, we construct $E$ by ``merging'' the roots of $E_1$ and $E_2$ --
that is to say, we remove the roots of $E_1$ and $E_2$ and add a new root, with
edges to all the vertices adjacent to either of the old roots; the new edge
labels are equal to the old edge labels.  The label of the new root is equal
to the product of the labels of the old roots.
\end{enumerate}
\end{defn}

See Figure~\ref{treeexamp} for an example illustrating this construction.

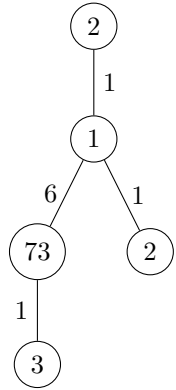
\begin{figure}
\caption{Low-defect tree for the expression $2((73(3x_1+1)x_2+6)(2x_3+1)x_4+1)$.}
\label{treeexamp}
\begin{center}
\begin{tikzpicture}
\node[circle,draw] {$2$}
	child {node[circle,draw] {$1$}
		child { node[circle,draw] {$73$}
			child {node[circle,draw] {$3$}
				edge from parent node[left] {$1$}}
			edge from parent node[left] {$6$}}
		child {node[circle,draw] {$2$}
			edge from parent node[right]{$1$}}
		edge from parent node[right]{$1$}};
\end{tikzpicture}
\end{center}
\end{figure}

We can use these trees to define a notion of equivalence for expressions:

\begin{defn}
Two low-defect expressions are said to be equivalent if their corresponding
trees are isomorphic.  (Here isomorphism must preserve both the root and all
labels.)
\end{defn}

Furthermore, every such tree occurs in this way:

\begin{prop}
\label{treeexist}
Every rooted tree, with vertices and edges labeled by positive integers, occurs
(up to isomorphism) as the tree for some low-defect expression.
\end{prop}

\begin{proof}
Call the tree $T$.  We prove this by induction on the number of vertices.  If
$T$ has only one vertex, the root, labeled $n$, it occurs as the tree for the
low-defect expression $n$.  Otherwise, the tree has more than one vertex, i.e.,
the root has at least one child.

If the root has only one child, let $T'$ be the tree obtained by deleting the
root of $T$, and let $E'$ be a low-defect expression that yields it.  If the
root is labeled $n$ and the unique edge off of it is labeled $c$, and $x$ is a
variable not appearing in $E'$, then the expression $n(E'\cdot x+c)$ is a
low-defect expression that yields $T$.  (If $n=1$, we may omit the
multiplication by $n$.)

Finally, the root could have more than one child; call its children
$v_1,\ldots,v_r$, and call its label $n$.  Then for $1\le i\le r$, let $T_i$ be
the tree obtained by removing all vertices except the root and the descendants
of $v_i$, and relabeling the root to have a label of $1$.  Then for each $i$ we
can pick a low-defect expression $E_i$ that yields $T_i$; then
the expression $n\cdot E_1\cdots E_r$ is a low-defect expression that
yields $T$.  (Again, if $n=1$, we may omit the multiplication by $n$.)
\end{proof}

Because of Proposition~\ref{treeexist}, we can use the term ``low-defect tree''
to simply refer to a rooted tree with vertices and edges labeled by positive
integers.  Also, among the various expressions in an equivalence class (i.e.,
that yield the same tree), the one constructed by Proposition~\ref{treeexist} is
one we'd like to pick out:

\begin{defn}
Given a low-defect tree $T$, a low-defect expression for it generated by the
method of Proposition~\ref{treeexist} (with multiplications by $1$ omitted) will
be called a \emph{reduced} low-defect expression for $T$.
\end{defn}

As mentioned above, passing from an expression $E$ to its tree $T$ loses a
little bit of information, but not very much.  We can, in fact, completely
characterize when two expressions will yield the same tree:

\begin{prop}
\label{tediousness}
Two low-defect expressions $E$ and $E'$ are equivalent if and only if one can
get from $E$ to the $E'$ by applying the following transformations to
subexpressions:
\begin{enumerate}
\item For low-defect expressions $E_1$ and $E_2$, one may replace $E_1\cdot E_2$
by $E_2 \cdot E_1$.
\item For low-defect expressions $E_1$, $E_2$, and $E_3$, one may replace
$(E_1\cdot E_2)\cdot E_3$ by $E_1\cdot (E_2\cdot E_3)$, and vice versa.
\item For integer constants $n$ and $m$, one may replace $n\cdot m$ by the
constant $nm$; and for an integer constant $k$ with $k=mn$, one may replace $k$
by $m\cdot n$.  This latter rule may only be applied if $k$ does not appear as
an addend in a larger expression.
\item For a low-defect expression $E_1$, one may replace $1\cdot E_1$ by $E_1$,
and vice versa.
\item One may rename all the variables in $E$, so long as distinct variables
remain distinct.  (This transformation can only be applied to $E$ as a whole,
not subexpressions.)
\end{enumerate}
\end{prop}

\begin{proof}
It's clear that all these moves do not change the tree.  The problem is proving
that all equivalences come about this way.

Suppose $T$ is the tree for $E$, $T'$ is the tree for $E'$, and $\phi:T\to T'$
is an isomorphism.  We induct on the number of vertices of $T$, the label of the
root, and the structure of $E$ and $E'$.

First we consider the case where either $E$ or $E'$ is a product.  In this case,
we decompose $E$ and $E'$ until we have written each as a product of low-defect
expressions which themselves are not products.  Each of these factors can either
be written as $F\cdot x+c$ for some low-defect expression $F$, some $x$ not
appearing in $F$, and some $c$; or as a natural number constant.  Say
$E=E_1\cdots E_r\cdot n_1\cdots n_s$ and
$E'=E'_1\cdots E'_{r'}\cdot n'_1\cdots n'_{s'}$, where the
$E_i$ and $E'_i$ have the former form and the $n_i$ are constants.  (Due to
rules (1) and (2), we do not need to worry about parenthesization or the order
of the factors.)  Note that by assumption, $r+s,r'+s'\ge 1$, and at least one of
them is at least $2$.
 
Let $T_i$ denote the tree of $E_i$ and $T'_i$ denote the tree of $E'_i$.  Then
we can conclude that the root of $T$ has $r$ children, and that $T_i$ can be
formed from $T$ by removing, along with all their descendants, all the children
of the root except child $i$, and changing the label of the root to $1$.
Similarly with $T'_i$ and the $r'$ children of its root.  Similarly, if we let
$N$ denote the product of the $n_i$, and $N'$ the product of the $n'_i$, we see
that $N$ is the label of the root of $T$, and $N'$ the label of the root of
$T'$.  Since $T$ and $T'$ are isomorphic, then, we have $N=N'$, $r=r'$, and
$\phi$ maps the children of the root of $T$ to the children of the root of $T'$.
This allows us to construct isomorphisms $\phi_i : T_i \to T'_{\sigma(i)}$,
where $\sigma$ is a fixed permutation in the group $S_r$.  By the inductive
hypothesis, then, each $E_i$ can be turned into $E'_{\sigma(i)}$ by use of moves
of type (1)-(5); we can then use rules (1) and (2) to put these back in the
original order.  (Note that rule (5) should be applied all at once, at the end,
so as to ensure that no two distinct variables are ever turned into the same
variable.)

Meanwhile, the product $n_1\cdots n_s$ may be turned into the product
$N$ by moves of type (3) and (4) (type (4) is necessary if $s=0$; note that in
this case we cannot have $r=0$).  But $N=N'$, which can be turned back into the
product $n'_1\cdots n'_{s'}$ by moves of type (3) and (4) as well.
This concludes the case where either $E$ or $E'$ is a product.

In the case where neither $E$ nor $E'$ is a product, $E$ can either be an
integer constant $n$, or it can be of the form $F\cdot x+c$, where $F$ is a
low-defect expression, $x$ is a variable not appearing in $F$, and $c$ is an
integer constant.  In the former case, $T$ has no non-root vertices, so neither
does $T'$; since we assumed $E'$ is not a product, this means it too is an
integer constant $n'$.  However, $n$ is the label of the unique vertex of $T$,
and $n'$ that of $T'$, and since $T\cong T'$, this implies $n=n'$.  Thus $E$ and
$E'$ are simply equal, and no moves need be applied.

Finally, we have the case where $E=F\cdot x+c$ as above.  In this case, we must
also be able to similarly write $E'=F'\cdot x'+c'$, as if $E'$ were a constant,
$E$ would be as well by the above argument.  Let $U$ and $U'$ denote the trees
of $F$ and $F'$, respectively.  Then $T$ consists of $U$ together with a new
root adjoined with a label of $1$, with the unique edge off of it labeled $c$;
and the relation between $T'$, $U'$, and $c'$ is the same.  Then since $T\cong
T'$, we conclude that $c=c'$ and $U\cong U'$.  By the inductive hypothesis,
then, $U$ may be transformed into $U'$ by moves of type (1)-(5); this transforms
$T$ from $F\cdot x+c$ to $F'\cdot y+c$, where $y$ is some variable not appearing
in $F'$.  (Since when applying rule (5), one may have to rename $x$ if one
changes one of the variables of $F$ to $x$.)  One may then apply rule (5) again
to replace $y$ by $x'$, completing the transformation into $E'$.  This proves
the proposition.
\end{proof}

This tells us also:
\begin{cor}
\label{ttof}
If $E_1$ and $E_2$ are equivalent low-defect expressions that both yield the
tree $T$, they also yield the same low-defect polynomial $f$, up to renaming of
the variables.  That is to say, up to renaming of the variables, it is possible
to determine $f$ from $T$.
\end{cor}

\begin{proof}
With the exception of renaming the variables, all of the moves allowed in
Proposition~\ref{tediousness} consist of replacing subexpressions with other
subexpressions that evaluate to the same thing.  This proves the claim.
\end{proof}

Note that inequivalent expressions (distinct trees) can also give rise to the
same polynomial; for instance, $2(2x+1)$ and $4x+2$ are inequivalent expressions
both yielding the polynomial $4x+2$ (see Figure~\ref{inequiv}).  However we will
see in Section~\ref{subsecorder} that from the polynomial $f$ we can recover at
least the ``shape'' of $T$, i.e., the isomorphism class of the rooted but
unlabeled tree underlying $T$.

\begin{figure}
\caption{Two different trees yielding the polynomial $4x+2$}
\label{inequiv}
\begin{center}
\begin{tikzpicture}
\node[circle,draw]{$2$}
	child {node[circle,draw]{$2$} edge from parent node[left] {$1$}};
\end{tikzpicture}
\qquad
\begin{tikzpicture}
\node[circle,draw]{$1$}
	child {node[circle,draw]{$4$} edge from parent node[left] {$2$}};
\end{tikzpicture}
\end{center}
\end{figure}
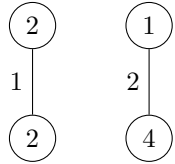

Now, the non-root vertices of the tree correspond to the variables of the
original expression:

\begin{defn}
Let $E$ be a low-defect expression and $T$ the corresponding tree.  We
recursively define a bijection between the variables of $E$ and the non-root
vertices of $T$ as follows:
\begin{enumerate}
\item If $E$ is an integer constant $n$, then it has no variables, and $T$ has
no non-root vertices, and the bijection is the trivial one.
\item If $E=E'\cdot x + c$, with $T'$ the tree for $E'$, then we use the
correspondence between variables of $E'$ and the non-root vertices of $T'$ to
associate variables of $E'$ with vertices of $T'\subseteq T$; and we assign the
root of $T'$ to correspond to the variable $x$.
\item If $E=E_1 \cdot E_2$, with $T_1$ and $T_2$ the trees for $E_1$ and $E_2$
respectively, then we use the correspondence between variables of $E_1$ and
non-root vertices of $T_1$ to associate variables of $E_1$ with vertices of
$T_1\subseteq T$; and we do similarly with $E_2$ and $T_2$.
\end{enumerate}
\end{defn}

See Figure~\ref{bijexamp} for an illustration of this bijection.

\begin{figure}
\caption{Low-defect tree for the expression $2((73(3x_1+1)x_2+6)(2x_3+1)x_4+1)$;
non-root vertices have been marked with corresponding variables in addition to
their labels.}
\label{bijexamp}
\begin{center}
\begin{tikzpicture}
\node[circle,draw] {$2$}
	child {node[circle,draw] {$1,x_4$}
		child { node[circle,draw] {$73,x_2$}
			child {node[circle,draw] {$3,x_1$}
				edge from parent node[left] {$1$}}
			edge from parent node[left] {$6$}}
		child {node[circle,draw] {$2,x_3$}
			edge from parent node[right]{$1$}}
		edge from parent node[right]{$1$}};
\end{tikzpicture}
\end{center}
\end{figure}
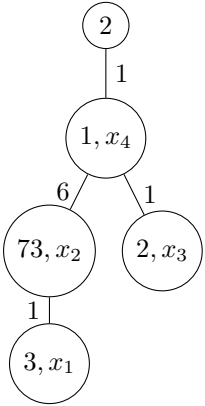

Equivalently, each variable can be thought of as corresponding to an edge rather
than to a non-root vertex; if the variable $x$ corresponds to the vertex $v$, we
can instead think of it as corresponding to the edge between $v$ and the parent
of $v$.  If we think of variables as corresponding to vertices, however, then we
can imagine the root as corresponding to the extra variable in the augmented
low-defect expression $\xpdd{E}$, although this analogy is not perfect.

This bijection, placing the variables of $E$ on the tree, shows us that the
variables of a low-defect expression do not all play the same role.  In
Section~\ref{sectrunc}, we will make extensive use of the variables
corresponding to leaves.  See also Remark~\ref{maxlrem} regarding the variables
corresponding to the children of the root.  In the following subsection, we will
begin to lay out the details of how this works.

\subsection{The nesting order, keys, and anti-keys}
\label{subsecorder}

Given a low-defect expression, we will define a partial order, the nesting
order, on its set of variables.  First, let us make the following observation:

\begin{prop}
Let $E$ be a low-defect expression.  Each variable of $E$ appears exactly once
in $E$, and there is a smallest low-defect subexpression of $E$ that contains
it.
\end{prop}

\begin{proof}
By definition, a variable of $E$ appears in $E$.  A variable of $E$ cannot
appear twice in $E$, as no rule of constructing low-defect expressions allows
this; rule (2) only allows multiplying two low-defect expressions if their
variables are disjoint, and rule (3) can only introduce a new variable different
from the ones already in $E$.

For the second part, observe that rule (3) is the only rule that introduces new
variables; so say $x$ is some variable of $E$, it must have been introduced via
rule (3).  This means that it occurs in a subexpression of $E$ of the form
$E'\cdot x+c$, where $E'$ is a low-defect expression and $c$ is a positive
integer constant.  Since $x$ itself is not a low-defect expression, and neither
is $E'\cdot x$, the next-smallest subexpression containing $x$, i.e., $E'\cdot
x+c$, is the smallest low-defect subexpression of $E$ that contains $x$.
\end{proof}

Because of this, it makes sense to define:

\begin{defn}
\label{defnest}
Let $E$ be a low-defect expression.  Let $x$ and $y$ be variables appearing in
$E$.  We say that $x\preceq y$ under the nesting ordering for $E$ if $x$ appears
in the smallest low-defect subexpression of $E$ that contains $y$.
\end{defn}

This is, in fact, a partial order:

\begin{prop}
The nesting ordering for a low-defect expression $E$ is a partial order.
\end{prop}

\begin{proof}
We have $x\preceq x$ as $x$ appears in any expression containing $x$.  If
$x\preceq y$ and $y\preceq z$, then the smallest low-defect expression
containing $z$ also contains $y$, and hence contains the smallest low-defect
expression containing $y$, and hence contains $x$.  And if $x\preceq y$ and
$y\preceq x$, then the smallest low-defect expression containing each is
contained in the other, i.e., the smallest low-defect expression containing $x$
is the smallest low-defect expression containing $y$.  Since the former has the
form $E_1 \cdot x + c_1$, and the latter has the form $E_1\cdot y+c_2$, we must
have $x=y$.
\end{proof}

In fact, it's not just any partial order -- it's a partial order that we've
already sort of seen; it's the partial order coming from the bijection between
variables of a low-defect expression $E$ and non-root vertices of its tree $T$.

\begin{prop}
\label{treeorder}
Let $E$ be a low-defect expression, and let $T$ be the corresponding tree.  Then
$x\preceq y$ under the nesting ordering if and only if the vertex in $T$
corresponding to $x$ is a descendant of the vertex in $T$ corresponding to $y$.
\end{prop}

\begin{proof}
We prove this by structural induction on $E$.  If $E$ is an integer constant,
then there are no variables and the statement is trivial.

In the case where $E=E'\cdot x+c$, say $T'$ is the tree corresponding to $E'$.
Suppose $x_1$ and $x_2$ are variables of $E$.  If $x_1$ and $x_2$ are both
variables of $E'$, then by the inductive hypothesis, $x_1\preceq x_2$ in the
nesting ordering in $E'$ if and only if the vertex corresponding to $x_1$ in
$T'$ is a descendant of that corresponding to $x_2$.  However, it is clear that
$x_1\preceq x_2$ in the nesting ordering of $E'$ if and only if $x_1\preceq x_2$
in the nesting ordering of $E$, since the smallest low-defect subexpression of
$E'$ containing $x_2$ is necessarily also the smallest low-defect subexpression
of $E$ containing $x_2$; and similarly with the corresponding vertices.
Hence the proposition is proved in this case.  Otherwise, we must have that one
of the variables is $x$ itself; say the variables are $x$ and $x'$.  But in that
case we automatically have that $x'\preceq x$, and the vertex for $x'$ is a
descendant of that of $x$.

This leaves the case where $E=E_1 \cdot E_2$; say $T_1$ and $T_2$ are the trees
corresponding to $E_1$ and $E_2$.  If $x_1$ and $x_2$ are both
variables of $E_1$, then by the inductive hypothesis, $x_1\preceq x_2$ in the
nesting ordering in $E_1$ if and only if the vertex corresponding to $x_1$ in
$T_1$ is a descendant of that corresponding to $x_2$; but as above, it does not
matter if we consider this in $E_1$ and $T_1$ or $E$ and $T$.  Similarly the
statement holds if $x_1$ and $x_2$ are both variables of $E_2$.  Finally, if
$x_1$ is a variable of $E_1$ and $x_2$ is a variable of $E_2$, then $x_1$ and
$x_2$ are incomparable in the nesting ordering, as the smallest low-defect
subexpression containing $x_1$ is contained in $E_1$ and hence does not contain
$x_2$, and vice versa; and, correspondingly, the corresponding vertices are
incomparable in $T$.
\end{proof}

Now, we've already seen (Corollary~\ref{ttof}) that it is possible to determine
the low-defect polynomial $f$ for a low-defect expression $E$ from its tree $T$.
In fact, not only is it possible to do so, but we can write down an explicit
description of the terms of $f$ in terms of $T$.  Specifically:

\begin{prop}
\label{ttofformula}
Let $T$ be a low-defect tree (say with root $v_0$) and $f$ the corresponding
low-defect polynomials after assigning variables to the non-root vertices of
$T$; let $x_v$ denote the variable corresponding to the vertex $v$.  Then
for a subset $S$ of $V(T)\setminus\{v_0\}$, the monomial $\prod_{v\in S}
x_v$ appears in $f$ in and only if the subgraph induced by $S\cup{v_0}$ is a
subtree of $T$.  Furthermore, its coefficient is given by
\[ \left( \prod_{v\in S\cup\{v_0\}} w(v) \right)
	\left( \prod_{\substack{\textrm{$e$ has exactly one} \\
	\textrm{vertex in $S\cup\{v_0\}$}}} w(e) \right). \]
The constant term corresponds to the subtree $\{v_0\}$, and the leading term is
the term corresponding to all of $T$.
\end{prop}

\begin{proof}
Let $E$ be a low-defect expression giving rise to $T$; we use structual
induction on $E$.  If $E$ is an integer constant $n$, then $T$ consists of just
a root labeled with $n$.  So the only rooted subtree of $T$ is $T$ itself,
containing no non-root vertices; and, correspondingly, $f$ has a unique term,
containing no variables, and with coefficient $n$, which matches the formula
given.

If $E=E'\cdot x+c$, say $T'$ and $f'$ are the tree and the polynomial arising
from $E'$.  Let $v_x$ be the vertex of $T$ corresponding to $x$, which is also
the root of $T'$.  Then a rooted subtree of $T$ consists of either just $v_0$,
or $v_0$ together with a rooted subtree of $T'$.  Correspondingly, since
$f=xf'+c$, a term of $f$ is either $x$ times a term of $f$, or just $c$.  The
subtree $\{v_0\}$ contains no non-root vertices and so corresponds to $c$; since
the root is labeled with a $1$ and the sole edge out of it is labeled with a
$c$, the formula for the coefficient is correct.  Any other rooted subtree $X$
consists of $v_0$ together with a rooted subtree $X'$ or $T'$; $X'$ corresponds
to some term $m'$ of $f'$.  Then we have a term $xm'$ in $f$, which corresponds
to $X$, since the old root of $T'$ is also the vertex $v_x$.  Furthermore, the
coefficient matches that given by the formula, changing $X'$ to $X$ just means
adding in the vertex $v_0$ and the edge $\{v_0,v_x\}$; however, $v_0$ has a
label of $1$, not changing the product, and the label of the edge $\{v_0,v_x\}$
is irrelevant as both vertices are in $X$.  (Moreover, no edges drop out of the
product, as the only new vertex is $v_0$, and its only edge is $\{v_0,v_x\}$.)
And since every term of $f$ is either $c$ or of the form $xm'$ for some term
$m'$ of $f'$, every term arises in this way.

This leaves the case where $E=E_1\cdot E_2$; say each $E_i$ gives rise to a
trees $T_i$ and a polynomial $f_i$, and let $v_i$ denote the root of $T_i$.
Then a rooted subtree of $T$ consists of $\{v_0\}$ together with subsets
$X_1\subseteq T_1$ and $X_2\subseteq T_2$ such $X_i\cup \{v_i\}$ is a rooted
subtree of $T_i$. Correspondingly, $f=f_1 f_2$, so each term of $f$ is the
product of a term of $f_1$ and a term of $f_2$; since $f_1$ and $f_2$ have no
variables in common, terms $m_1m_2$ are determined uniquely by the pair
$(m_1,m_2)$, which by the inductive hypothesis are in bijection with sets
$(X_1,X_2)$ as described above.  It remains to check that the coefficients
match.  Say $X_1$ and $X_2$ are subsets as described above, with each $X_i$
corresponding to a term $m_i$ of $f_i$, so that the subtree $X_1\cup
X_2\cup\{v_0\}$ corresponds to the term $m_1m_2$.  Then the product of the
labels of vertices in $X_1 \cup X_2 \cup \{v_0\}$ is the product of the labels
of vertices in $X_1\cup X_2$ times $w(v_0)$, the latter of which is equal to
$w(v_1)w(v_2)$, so this is the same as the product of the labels of vertices in
$X_1\cup \{v_1\}$ times the product of the labels of vertices in
$X_2\cup\{v_2\}$.  Meanwhile, the product over the edges is also the product of
both the previous ones, as the only edges that could change are those that
connected $X_1$ to $v_1$ or $X_2$ to $v_2$, all of which were previously not in
the product due to having both vertices in one of the $X_i\cup\{v_i\}$; but
these now connect $X_1$ and $X_2$ to $v_0$, with both vertices in $X_1\cup
X_2\cup\{v_0\}$, so they still are not in the product.

Finally, the leading term corresponds to all of $T$ as it contains all the
variables, and the constant term corresponds to $\{v_0\}$ as it contains none of
the variables.
\end{proof}

This yields the following corollary, which will be useful in
Section~\ref{sectrunc}:

\begin{cor}
\label{excludedminl}
Let $E$ be a low-defect expression and $f$ the corresponding low-defect
polynomial.  Any term of $f$ other than the leading term must exclude at least
one minimal variable.
\end{cor}

\begin{proof}
Consider the low-defect tree corresponding to $E$.  Any subtree other than the
whole tree must exclude at least one leaf, i.e., the corresponding term of $f$
must exclude at least one minimal variable.
\end{proof}

It also, in particular, tells us the leading coefficient of $f$ in terms of the
$T$, which we will use in Section~\ref{bdsec}:

\begin{cor}
\label{leadcoefftree}
Let $T$ be a low-defect tree, and $f$ be the corresponding low-defect
polynomial.  Then the leading coefficient of $f$ is the product of the vertex
labels of $T$.
\end{cor}

\begin{proof}
The leading term corresponds to the subtree consisting of all of $T$.  This
includes all the vertices; and no edge has exactly one vertex in it, as all
edges have both vertices in it.
\end{proof}

Now, as we've already noted above, we cannot go backwards from $f$ to determine
$T$; the map from trees to polynomials is not one-to-one.  However, we can go
part of the way back -- we can determine the ``shape'' of $T$, that is to say,
the isomorphism class of the rooted but unlabeled tree underlying $T$; it is
only the labels we cannot determine with certainty.

To do this, for a low-defect polynomial $f$, consider the set of monomials that
appear in $f$, without their associated coefficients; ignoring the nesting
ordering for a moment, these monomials can be partially ordered by divisibility.
But we can, in fact, recover the nesting ordering (and thus the shape of $T$,
without labels) from this partial ordering.  First, a definition:

\begin{defn}
Let $E$ be a low-defect expression yielding a low-defect tree $T$ and a
low-defect polynomial $f$; let $x$ be a variable in $E$ and $v_x$ the
corresponding vertex in $T$.  We define the \emph{key} of $x$ in $E$ to be the
term of $f$ corresponding to the subtree consisting of all ancestors of $v_x$.
We define the \emph{anti-key} of $x$ in $E$ to be the term of $f$ corresponding
to the subtree consisting of all non-descendants of $v_x$.  So the key of $x$ is
the smallest term of $f$ containing $x$ (under divisibility ignoring
coefficients), and the anti-key of $x$ is the largest term not containing $x$.
\end{defn}

Both these operations, key and anti-key, are order-reversing:

\begin{prop}
\label{recoverorder}
Let $E$ be a low-defect expression, and let $x$ and $y$ be variables appearing
in $E$.  Then $x\preceq y$ under the nesting ordering if and only if the key of
$y$ divides the key of $x$ (ignoring coefficients), which also occurs if and
only if the anti-key of $y$ divides the anti-key of $x$.
\end{prop}

\begin{proof}
Let $T$ be the low-defect tree determined by $E$, and let $v_x$ and $v_y$ be the
vertices corresponding to $x$ and $y$.  By Proposition~\ref{treeorder},
$x\preceq y$ if and only if, $v_x$ is a descendant of $v_y$.  But if $v_x$ is a
descendant of $v_y$, then every ancestor of $v_y$ is an ancestor of $v_x$, and
so (ignoring coefficients), the key of $y$ divides the key of $x$.  Conversely,
if the key of $y$ divides the key of $x$, then $y$ divides the key of $x$, and
so $v_y$ is an ancestor of $v_x$.  Similarly, if $v_x$ is a descendant of $v_y$,
then every non-descendant of $v_y$ is a non-descendant of $v_x$, and so the
anti-key of $y$ divides the anti-key of $x$ (ignoring coefficients).  Convesely,
if the anti-key of $y$ divides the anti-key of $x$, then every non-descendant of
$v_y$ is a non-descendant of $v_x$, i.e., every descendant of $v_x$ is a
descendant of $v_y$, i.e., $v_x$ is a descendant of $v_y$ and so $x\preceq y$.
\end{proof}

Thus, from $f$ alone, the nesting ordering on the variables can be recovered;
for as we saw above, it is possible from $f$ alone to determine the key and the
anti-key of some variable in $f$ (so we can speak simply of ``the key of $x$ in
$f$'', or ``the anti-key of $x$ in $f$'').  But by
Proposition~\ref{recoverorder}, if $x$ and $y$ are variables in $f$, and we know
their keys or anti-keys, we can determine whether or not $x\preceq y$, without
needing to know the tree or expression that $f$ came from; it does not depend on
those things.  Thus it makes sense to simply talk about the nesting ordering on
the variables of $f$.  Furthermore this means we can also recover the shape of
$T$ from $f$ alone; the vertex corresponding to $x$ is a child of the vertex
corresponding to $y$ if and only if $x\preceq y$ and there are no other
variables inbetween, and the vertex corresponding to $x$ is a child of the root
if and only if $x$ is maximal in the nesting ordering.

Indeed, we can, given $f$, determine all trees $T$ that yield it.  By above, we
know the shape, and which variables correspond to which vertices, and
Proposition~\ref{ttofformula} constrains the vertex and edge labels -- indeed,
it not only constrains them, it bounds them (as every label divides at least one
coefficient of $f$), making it possible to determine all $T$ that yield $f$ (and
thus to determine $\cpx{f}$ via brute-force search.  (One can also use this
procedure to determine if $f$ is a low-defect polynomial at all, if one does not
already know.)  But this is rather more involved than what is needed to compute
the complexity of a low-defect expression or tree!

\begin{rem}
\label{maxlrem}
It is the minimal variables of $f$ will turn out to be quite important in
Section~\ref{sectrunc}, but it's worth noting that the maximal variables have a
use too -- in \cite{paperwo}, the proposition was proved (Lemma~4.3) that if $f$
is a low-defect polynomial of degree at least $1$, there exists a variable $x$,
low-defect polynomials $g$ and $h$, and a positive integer $c$ such that
$f=h\cdot (g\cdot x+c)$.  With this framework -- if we allow for the use of
commutativity and associativity -- we can easily see that these $x$ are
precisely the maximal variables of $f$.
\end{rem}

\subsection{A lower bound on the complexity of a low-defect polynomial}
\label{bdsec}

In this section, we will discuss the notion of the complexity of a low-defect
expression, tree, or polynomial, and use this to prove a lower bound on the
complexity of a low-defect polynomial (Corollary~\ref{polycpxbd}).  This lower
bound is what allows us to show that truncation will keep the degrees of our
polynomials low despite our use of small step sizes (see discussion in
Section~\ref{buildsec}).

A low-defect expression has an associated base complexity:

\begin{defn}
We define the complexity of a low-defect expression $E$, denoted $\cpx{E}$, as
follows:
\begin{enumerate}
\item If $E$ is a positive integer constant $n$, we define $\cpx{E}=\cpx{n}$.
\item If $E$ is of the form $E_1 \cdot E_2$, where $E_1$ and $E_2$ are
low-defect expressions, we define $\cpx{E}=\cpx{E_1}+\cpx{E_2}$.
\item If $E$ is of the form $E' \cdot x + c$, where $E'$ is a low-defect
expression, $x$ is a variable, and $c$ is a positive integer constant, we define
$\cpx{E}=\cpx{E'}+\cpx{c}$.
\end{enumerate}
\end{defn}

In Section~\ref{review} we defined $\cpx{f}$, for a low-defect polynomial $f$,
to be the smallest $C$ such that $(f,C)$ is a low-defect pair.  Above, we
also defined the notion of $\cpx{E}$ for $E$ a low-defect expression.  These are
compatible as follows:

\begin{prop}
Let $f$ be a low-defect polynomial.  Then $\cpx{f}$ is the smallest value of
$\cpx{E}$ among low-defect expressions $E$ that evaluate to $f$.
\end{prop}

\begin{proof}
The rules for building up a low-defect pair $(f,C)$ are exactly the same as the
rules for building a low-defect expression $E$, and what these rules do to the
base complexity $C$ is exactly the same as what they do to the complexity
$\cpx{E}$ (except that they allow for increasing $C$ further).  So each
low-defect pair $(f,C)$ comes from some low-defect expression $E$ yielding $f$
with $\cpx{E}\le C$, and any low-defect expression $E$ yielding $f$ yields a
low-defect pair $(f,\cpx{E})$.  So the lowest possible value of $C$ and of
$\cpx{E}$ are the same.
\end{proof}

Indeed, though we will not use this formalism here, it may make sense to
consider ``low-defect expression pairs'', pairs $(E,C)$ where $E$ is a
low-defect expression and $C\ge \cpx{E}$.  After all, the definition of
$\cpx{E}$ assumes one knows the complexities of the integer constants appearing
in $\cpx{E}$, but one may not know these exactly, but only have an upper bound
on them.  For instance, one might not be using low-defect expressions as we
defined them here, but rather ones where, instead of integer constants, one has
representations of integers in terms of $1$, $+$, and $\cdot$.  That is to say,
perhaps one is not using expressions such as $2(2x+1)$, but rather such as
$(1+1)((1+1)x+1)$.  In this example, the expressions used for the integer
constants were most-efficient, but this may not be the case in general.  In this
case, it would make sense to consider the complexity of the expression to be
simply the number of $1$'s used, which would be an upper bound on the complexity
of the low-defect expression it yields.  This sort of only having an upper bound
is, after all, the reason we consider pairs $(f,C)$, and it may make sense in
other contexts to do with expressions as we do here with polynomials.

Since we like to encode low-defect expressions as trees, it makes sense to
define the complexity of these:

\begin{defn}
The complexity of a low-defect tree, $\cpx{T}$, is defined to be the smallest
$\cpx{E}$ among all low-defect expressions yielding $T$.
\end{defn}

Note that it follows from this definition that for a low-defect polynomial $f$,
$\cpx{f}$ can be equivalently characterized as the smallest $\cpx{T}$ among all
trees $T$ yielding $f$.  Again, it may make sense in other contexts to consider
pairs $(T,C)$ with $C\ge \cpx{T}$, for the same reasons discussed above.  If,
however, we do know the complexity of arbitrary natural numbers, then the
complexities of expressions and of trees can be computed as follows:

\begin{prop}
\label{exprcpxbd}
We have:
\begin{enumerate}
\item Let $E$ be a low-defect expression.  Then $\cpx{E}$ is equal to the sum of
the complexities of all the integer constants occurring in $E$.
\item Let $T$ be a low-defect tree.  Then
\[ \cpx{T} = \sum_{e\ \textrm{an edge}} \cpx{w(e)} +
	\sum_{v\ \textrm{a leaf}} \cpx{w(v)}
	+ \sum_{\substack{v\ \textrm{a non-leaf vertex}\\ w(v)>1}} \cpx{w(v)},\]
where $w$ denotes the label of the given vertex or edge.
\end{enumerate}
\end{prop}

\begin{proof}
The first statement is a straightforward structural induction.  If $E$ is a
constant, its complexity is the complexity of that constant.  If $E=E_1 \cdot
E_2$, its complexity is $\cpx{E_1}+\cpx{E_2}$, which by the inductive hypothesis
is the sum of the complexities of all the constants used in either.  And if
$E=E'\cdot x+c$, its complexity is $\cpx{E'}+\cpx{c}$, which by the inductive
hypothesis is the sum of the complexities of the constants used in $E'$ plus
that of the new constant introduced.

For the second statement, consider a reduced low-defect expression $E$ giving
rise to $T$.  Then the edge and vertex labels correspond exactly to the
constants used in $E$, with the exception of labels of $1$ on non-leaf vertices.
As $\cpx{T}\le\cpx{E}$, this shows that the formula above is an upper bound on
$\cpx{T}$.  For the lower bound, note that by Proposition~\ref{tediousness}, any
other low-defect expression for $T$ can be obtained by $E$ by the listed moves.
Moves of the form (1), (2), and (5) do not alter the complexity of an expression
at all.

This leaves moves of type (3) and (4).  Suppose (3) or (4) is going to be
applied to a subexpression $E'$; consider $E'$ as a product (possibly of one
thing) and consider the largest product $P$ containing the factors of $E'$ as
factors.  That is to say, let $P$ be the largest subexpression of the form $E_1
\cdot \ldots \cdot E_k$ (where due to (1) and (2), we do not need to worry about
parenthesization or order) where the $E_i$ cannot be written as products, and
the factors of $E'$ are among the $E_i$.  Since (3) and (4), applied to factors
of $P$, only alter things within $P$, and do not alter the internals of any
$E_i$ which can be written as a sum, we see that the least complexity is
obtained by minimizing the complexity of each individual product $P$.  But this
is clearly done by multiplying together all constants and eliminating $1$'s
where possible.  This leaves us with an expression which is the same as $E$ up
to moves of the form (1), (2), and (5).  Hence $E$ has the lowest complexity
among expressions for $T$, and so $\cpx{T}=\cpx{E}$, which as noted, is given by
the formula.

\end{proof}

With this, we now obtain our lower bound:

\begin{prop}
\label{polycpxbd}
Let $(f,C)$ be a low-defect pair of degree $k$, and suppose that $a$ is the
leading coefficient of $f$.  Then $C\ge \cpx{a} + k$.  Equivalently, if $f$ is a
low-defect polynomial of degree $k$ with leading coefficient $a$, then
$\cpx{f}\ge \cpx{a}+k$.
\end{prop}

\begin{proof}
Let $T$ be a low-defect tree giving rise to $f$ with $C\ge \cpx{T}$.  Then
\begin{eqnarray*}
\cpx{T} \ge \sum_{e\ \textrm{an edge}} \cpx{w(e)} +
	\sum_{\substack{v\ \textrm{a vertex}\\ w(v)>1}} \cpx{w(v)} \\
\ge \left(\sum_{e\ \textrm{an edge}} 1 \right)+
	\left\| \prod_{\substack{v\ \textrm{a vertex}\\ w(v)>1}} w(v)\right\|.
\end{eqnarray*}
That is, applying Corollary~\ref{leadcoefftree}, it is at least the number of
edges plus $\cpx{a}$.  Since the number of edges is one less than the number of
vertices, the number of edges is $k$.  So $C\ge \cpx{a}+k$.

The second statement then follows as $\cpx{f}$ is by definition the smallest $C$
among low-defect pairs $(f,C)$.
\end{proof}

In particular, the degree of a polynomial is bounded by its defect:

\begin{cor}
\label{polydftbd}
Let $(f,C)$ be a low-defect pair of degree $k$, and suppose that $a$ is the
leading coefficient of $f$.  Then $\dft(f,C)\ge \dft(a)+k\ge k$.  Equivalently,
$\dft(f)\ge \dft(a)+k\ge k$.
\end{cor}

\begin{proof}
By definition, $\dft(f,C)=C-3\log_3 a$.  So
\[ \dft(f,C) = C-3\log_3 a \ge \cpx{a} + k - 3\log_3 a = \dft(a) + k,\]
and $\dft(a)+k\ge k$.  The second statement then follows as $\dft(f)$ is just
the smallest value of $\dft(f,C)$ among low-defect pairs $(f,C)$.
\end{proof}

\section{The truncation operation}
\label{sectrunc}

Now, finally, we can describe the operation of truncating a low-defect
polynomial (or expression, or tree) to a given defect -- the ``filtering-down''
half of our method.  The results here will be phrased in terms of low-defect
pairs, but the analogues for low-defect expressions are clear.

\subsection{Truncations and their properties}

First we just describe truncating a low-defect polynomial in general:

\begin{prop}
\label{dirtrunc}
Let $(f,C)$ be a low-defect pair, say of degree $r$, and suppose $x_i$ is a
variable of $f$ which is minimal with respect to the nesting ordering.  Let
$k\ge0$ be an integer, and define
\[ g(x_1,\ldots,x_{i-1},x_{i+1},\ldots,x_r) :=
	f(x_1,\ldots,x_{i-1},3^k,x_{i+1},\ldots,x_r).\]
Then:
\begin{enumerate}
\item The polynomial $g$ is a low-defect polynomial, and $(g,C+3k)$ is a
low-defect pair.
\item If $a$ is the leading coefficient of $f$, then the leading coefficient of
$g$ is strictly greater than $a3^k$, and so $\dft(g,C+3k)<\dft(f,C)$.
\item The nesting order on the variables of $g$ is the restriction of the
nesting order on the variables of $f$ to
$\{x_1,\ldots,x_{i-1},x_{i+1},\ldots,x_r\}$.
\item For any $k_1,\ldots,k_{i-1},k_{i+1},\ldots,k_r$, we have
\[ \dft_{g,C+3k}(k_1,\ldots,k_{i-1},k_{i+1},\ldots,k_r) =
	\dft_{f,C}(k_1,\ldots, k_{i-1}, k, k_{i+1}, \ldots, k_r).\]
\end{enumerate}
\end{prop}

\begin{proof}
Let $E$ be a low-defect expression of complexity at most $C$ giving rise to $f$;
we apply structural induction to prove parts (1), (2), and (3).  Note that $E$
cannot be an integer constant as then it would have no variables.

If $E=E'\cdot x+c$, there are two cases; either $E'$ has degree $0$, or it has
positive degree.  In the former case, $x$ is the unique minimal variable, so
$x_i=x$; say $E'$ evaluates to the constant $n$ and has complexity at most
$C'=C-\cpx{c}$.  Then $g$ is equal to the constant $n3^k+c$, which can be given
by a low-defect expression.  Furthermore, the complexity of this low-defect
expression is at most $C'+3k+\cpx{c}=C+3k$, so $(g,C+3k)$ is a low-defect pair.
And whereas the leading coefficient of $f$ was $n$, the leading coefficient
of $g$ is $n3^k+c>n3^k$.  Finally, $g$ has no variables, so part (3) is
trivially true.

Otherwise, if $E'$ has positive degree, then $x$ is not minimal, and the minimal
variables in $E$ are precisely the minimal variables in $E'$.  Assume without
loss of generality that $x=x_r$.  Say $E'$ has complexity at most
$C'=C-\cpx{c}$.  Let $f'$ be the polynomial coming from $E'$, and
\[ g'(x_1,\ldots,x_{i-1},x_{i+1},\ldots,x_{r-1}) :=
	f'(x_1,\ldots,x_{i-1},3^k,x_{i+1},\ldots,x_{r-1}).\]
Then by the inductive hypothesis, $g'$ is a low-defect polynomial, coming from
some low-defect expression $E''$ with complexity at most $C'+3k$.  So $g$ is a
low-defect polynomial as it comes from the low-defect expression $E''\cdot x+c$,
which has complexity at most
\[C'+3k+\cpx{c}=C+3k.\]  And if $a$ is the leading
coefficient of $f$, then it is also the leading coefficient of $f'$, and so by
the inductive hypothesis the leading coefficient of $g'$ is greater than $a3^k$,
but the leading coefficient of $g$ is the same as that of $g'$.  Finally, by the
inductive hypothesis, the nesting order on the variables of $g'$ is the
restriction of the nesting order of the variables of $f'$, and the nesting order
on the variables of $g$ is the same as that on the variables of $g'$, but with
$x_r$ added as a new maximum element; since the same relation holds between the
nesting order for $f$ and the nesting order for $f'$, part (3) is true in this
case.

This leaves the case where $E=E_1\cdot E_2$.  In this case, a minimal variable
of $E$ is either a minimal variable of $E_1$ or a minimal variable of $E_2$.
Suppose without loss of generality that
\[ E(x_1,\ldots,x_r)=E_1(x_1,\ldots,x_s)E_2(x_{s+1},\ldots,x_r)\]
and $i\le s$.  Say $E_1$ and $E_2$ give rise to polynomials $f'$ and $h$, and
$E_1$ has complexity at most $C'=C-\cpx{E_2}$.  Then if we define
\[ g'(x_1,\ldots,x_{i-1},x_{i+1},\ldots,x_s) :=
	f'(x_1,\ldots,x_{i-1},3^k,x_{i+1},\ldots,x_s),\]
by the inductive hypothesis, $g'$ is a low-defect polynomial, coming from some
low-defect expression $E'$ with complexity at most $C'+3k$.  So $g$ is a
low-defect polynomial as it comes from the low-defect expression $E'\cdot E_2$,
which has complexity at most $C'+3k+\cpx{E_2}=C+3k$.  And if $a_1$ is the
leading coefficient of $f'$ and $a_2$ is the leading coefficient of $h$, then
the leading coefficient of $f=f'\cdot h$ is $a_1 a_2$, while by the inductive
hypothesis, the leading coefficient of $g$ is strictly greater than $3^k a_1$,
and so the leading coefficient of $g=g'\cdot h$ is strictly greater than $3^k
a_1 a_2$.  Finally, the nesting order on the variables of $g$ is just the
disjoint union of the nesting order on the variables of $g'$ and the nesting
order on the variables of $h$, and the same relation holds between the nesting
order for $f$ and the nesting order for $f'$.  By the inductive hypothesis, the
nesting order for $g'$ is just the restriction of that for $f'$, so the same
relation holds between $g$ and $f$.

To prove the second statement in part (2), we note that if $a$ is the leading
coefficient of $f$ and $b$ is the leading coefficient of $g$, since $b>a3^k$,
\[
\dft(g,C+3k) = C+3k - 3\log_3(b) = C-3\log_3(b3^{-k})
	< C-3\log_3(a) = \dft(f,C).
\]

Finally, part (4) follows as
\begin{multline*}
\dft_{g,C+3k}(k_1,\ldots,k_{i-1},k_{i+1},\ldots,k_r) = \\
C+3k+3(k_1+\ldots+k_{i-1}+k_{i+1}+\ldots+k_r) -
	3\log_3 g(3^{k_1},\ldots,3^{k_{i-1}},3^{k_{i+1}},\ldots,3^{k_r})
= \\ C+3(k_1+\ldots+k_{i-1}+k+k_{i+1}+\ldots+k_r) -
	3\log_3 f(3^{k_1},\ldots,3^{k_{i-1}},3^k,3^{k_{i+1}},\ldots,3^{k_r})
= \\ \dft_{f,C}(k_1,\ldots,k_{i-1},k,k_{i+1},\ldots,k_r).
\end{multline*}

\end{proof}

\begin{defn}
Let $(f,C)$ be a low-defect pair, and let $(g,D)$ be obtained from it as in
Proposition~\ref{dirtrunc}; we will call $(g,D)$ a \emph{direct truncation} of
$f$, and $g$ an direct truncation of $f$.

Furthermore, we will define $(g,D)$ to be a \emph{truncation} of $(f,C)$ if
there are low-defect pairs $(f,C)=(f_0,C_0),(f_1,C_1),\ldots,(f_k,C_k)=(g,D)$
with $(f_{i+1},C_{i+1})$ a direct truncation of $(f_i,C_i)$.  Similarly in this
case we say $g$ is a truncation of $f$.
\end{defn}

Immediately we get:

\begin{prop}
\label{truncind}
Say $(f,C)$ is a low defect pair and $(g,D)$ is a truncation of it.  Then:
\begin{enumerate}
\item $\dft(g,D)<\dft(f,C)$.
\item The nesting order on the variables of $g$ is the restriction of the
nesting order on the variables of $f$.
\end{enumerate}
\end{prop}

\begin{proof}
This follows immediately from iterating parts (2) and (3) of
Proposition~\ref{dirtrunc}.
\end{proof}

So, when we truncate $f$, we are substituting powers of $3$ into some of the
variables, and leaving the other variables free.  Say $f$ has degree $r$, and
consider the function
\[ (k_1,\ldots,k_r) \mapsto f(3^{k_1},\ldots,3^{k_r}) \]
from $\mathbb{Z}_{\ge0}^r$ to $\mathbb{N}$; when we truncate $f$, we are fixing
the values of some of the $k_i$.  In a sense, we are restricting $f$ to a subset
of $\mathbb{Z}_{\ge0}^r$ fo the form $S_1\times\ldots\times S_r$, where each
$S_i$ is either a single point or all of $\mathbb{Z}_{\ge0}$.

As such we will want a way of talking about such sets; we will represent them by
elements of $(\mathbb{Z}_{\ge0}\cup\{*\})^r$, where here $*$ is just an abstract
symbol which is distinct from any whole number; it represents ``this position
can be any number'', or the set $\mathbb{Z}_{\ge0}$, where putting in an actual
number $n$ would represent ``this position must be $n$'', or the set $\{n\}$.
Let us formally define our way of getting a set from such an object, how we can
substitute these objects into low-defect polynomials:

\begin{defns}
Given $(k_1,\ldots,k_r)\in(\mathbb{Z}_{\ge0}\cup \{*\})^r$, we define
$S(k_1,\ldots,k_r)$ to be the set
\[\{(\ell_1,\ldots,\ell_r)\in\mathbb{Z}_{\ge0}^r : \ell_i=k_i\ \textrm{for}\
k_i\ne *\}.\]

Furthermore, given $f\in\mathbb{Z}[x_1,\ldots,x_r]$, we define the
$3$-substitution of $(k_1,\ldots,k_r)$ into $f$ to be the polynomial obtained by
substituting $3^{k_i}$ for $x_i$ whenever $k_i\ne *$.  If $(f,C)$ is a
low-defect pair, we define the $3$-substitution of $(k_1,\ldots,k_r)$ to be
$(g,D)$ where $g$ is the $3$-substitution of $(k_1,\ldots,k_r)$ into $f$, and
$D=C+3\sum_{k_i\ne *} k_i$.
\end{defns}

Be warned that in general, $3$-substituting into a low-defect pair may not yield
a low-defect pair, if one substitutes into the wrong variables.  For instance,
if $(f,C)=((3x_1+1)x_2+1,5)$, then $3$-substituting in $(*,1)$ yields
$(9x+4,8)$, which is not a low-defect pair.  And if
$(f,C)=((3x_1+1)(3x_2+1)x_3+1,9)$, and one $3$-substitutes in $(*,*,0)$, then
one obtains $(9x_1x_2+3x_1+3x_2+2,9)$, the first element of which is not a
low-defect polynomial at all.

However, in what follows, we will only be using this notion in cases where it
does, in fact, turn out to be a low-defect pair.  Specifically, in the following
cases:

\begin{prop}
Let $(f,C)$ be a low-defect pair, and let
$(k_1,\ldots,k_r)\in(\mathbb{Z}_{\ge0}\cup\{*\})^r$ be such that the set of $i$
for which $k_i\ne *$ corresponds to a downward-closed subset of the variables of
$f$.  Let $(g,D)$ denote the $3$-substitution of $(k_1,\ldots,k_r)$ into
$(f,C)$.  Then:
\begin{enumerate}
\item The pair $(g,D)$ is a truncation of $(f,C)$ (and hence a low-defect pair).
\item Let $t$ be the number of $i$ such that $k_i=*$, and let $\iota$ be the map
from $\mathbb{Z}_{\ge0}^t$ to $\mathbb{Z}_{\ge0}^r$ given by inserting the
arguments $(\ell_1,\ldots,\ell_t)$ into the coordinates of $(k_1,\ldots,k_r)$
where $k_i=*$.  Then $\dft_{g,D}=\dft_{f,C}\circ \iota$.
\end{enumerate}
Furthermore, all truncations of $(f,C)$ arise in this way.
\end{prop}

\begin{proof}
We first prove part (1).  Let $i_1, \ldots, i_s$ be the indices for which
$k_i\ne *$, enumerated in an order such that if $x_{i_j} \preceq x_{i_{j'}}$
then $i_j \le i_{j'}$.  Let $(f_0,C_0)=(f,C)$.  Now, for $1\le j\le s$, given
$(f_{j-1},C_{j-1})$, we will take $(f_j,C_j)$ to be the direct truncation of
$(f_{j-1},C_{j-1})$ where $3^{k_{i_j}}$ is substituted into $x_{i_j}$.  Of
course, in order for this to be a direct truncation, $x_{i_j}$ must be minimal
in $f_{j-1}$.  But this follows due to the order we have enumerated the
elements; by assumption, each $x_{i_j}$ is minimal in
$\{x_{i_j},\ldots,x_{i_s}\}$, and since $\{x_{i_1},\ldots,x_{i_s}\}$ is
downwardly closed in $\{x_1,\ldots,x_r\}$, we have that
$\{x_{i_j},\ldots,x_{i_s}\}$ is downwardly closed in
$\{x_1,\ldots,x_r\}\setminus\{x_{i_1},\ldots,x_{i_{j-1}}\}$, and so $x_{i_j}$ is
minimal in $\{x_1,\ldots,x_r\}\setminus\{x_{i_1},\ldots,x_{i_{j-1}}\}$.  And by
Proposition~\ref{truncind}, this last set is precisely the set of variables of
$f_j$, with the same nesting order.  Thus this is indeed a truncation.

Part (2) follows by simply iterating part (4) of Proposition~\ref{dirtrunc} in
the above.  Finally, we can see that every truncation arises in this way by
inducting on the number of steps in the truncation.  If there are no steps, then
this is true with $(k_1,\ldots,k_r)=(*,\ldots,*)$.  Otherwise, say that
$(f_s,C_s)$ is an $s$-step truncation of $(f,C)$ and that $(f_{s+1},C_{s+1})$ is
a direct truncation of it; we assume by induction that $(f_s,C_s)$ is the
$3$-substitution into $(f,C)$ of some tuple
$(k_1,\ldots,k_r)\in(\mathbb{Z}_{\ge0}\cup\{*\})^r$.  Then $(f_{s+1},C_{s+1})$
is the $3$-substitution into $(f_s,C_s)$ of some tuple
$(*,\ldots,*,\ell_j,*,\ldots,*)\in(\mathbb{Z}_{\ge0}\cup\{*\})^{r-s}$ (here
$\ell_j\ne *)$.  This makes it the $3$-substitution into $(f,C)$ of some tuple
$(k'_1,\ldots,k'_r)$, where $k'_i=k_i$ when $k_i\ne *$, and $k_i=\ell_j$ for one
particular $i$ with $k_i=*$.
\end{proof}

So, in fact, we'll only be using $3$-substitution in cases where it yields a
truncation; or, really, we'll just be using it as another way of thinking about
truncation.

\subsection{Truncating a polynomial to a given defect}
\label{polytruncsec}

Having discussed truncation in general, we can now discuss how to truncate a
low-defect polynomial to a given defect.  Earlier, in Proposition~\ref{supdfts},
we showed that for a low-defect pair $(f,C)$, the number $\dft(f,C)$ is the
least upper bound of the values of $\dft_{f,C}$.  Now we show that something
stronger is true:

\begin{prop}
\label{minlimit}
Let $(f,C)$ be a low-defect pair of degree $r$.  Say $x_{i_j}$, for $1\le j\le
s$, are the minimal variables of $f$.  Then
\[ \lim_{k_{i_1},\ldots,k_{i_s}\to\infty}
	\dft_{f,C}(k_1,\ldots,k_r) = \dft(f,C)\]
(where the other $k_i$ remain fixed).
\end{prop}

\begin{proof}
Consider once again $g$, the reverse polynomial of $f$:
\[g(x_1,\ldots,x_r)=x_1\ldots x_r f(x_1^{-1},\ldots,x_r^{-1}).\]
So $g$ is a multilinear polynomial in $x_1, \ldots, x_r$, with the coefficient
of $\prod_{i\in S} x_i$ in $g$ being the coefficient of $\prod_{i\notin S} x_i$
in $f$.  Let $a$ denote the leading coefficient of $f$, which is also the
constant term of $g$.

By Corollary~\ref{excludedminl}, every non-leading term of $f$ excludes some
minimal variable.  Hence every non-constant term of $g$ includes some minimal
variable.  So if we once again write
\[ \dft_{f,C}(k_1,\ldots,k_r)= C-3\log_3 g(3^{-k_1},\ldots,3^{-k_r}), \]
we see that as the minimal variables approach infinity, then each non-constant
term of $g(3^{-k_1},\ldots,3^{-k_r})$ approaches $0$, and so
$g(3^{-k_1},\ldots,3^{-k_r})$ approaches $a$.  So once again we have

\[ \lim_{k_{i_1},\ldots,k_{i_s}\to\infty} \dft_{f,C}(k_1,\ldots,k_r)=
	C-3\log_3 a=\dft(f,C).\]
\end{proof}

One can obtain numerical versions of this proposition, but we do not bother to
state them here.

We can restate this proposition as follows:

\begin{cor}
\label{dirtrunc2}
Let $(f,C)$ be a low-defect pair, say of degree $r>0$, let $0\le s<\dft(f,C)$ be
a real number.
Then there exists a number $K$ such that, whenever
$\dft_{f,C}(k_1,\ldots,k_r)<s$, then $k_i\le K$ for some $i$ such that $x_i$ is
minimal in the nesting ordering for $f$.
\end{cor}

\begin{proof}
By Proposition~\ref{minlimit}, since $s<\dft(f,C)$, we can choose some $K$ such that $\dft_{f,C}(k_1,\ldots,k_r)\ge s$, where
$k_i=K+1$ if $x_i$ is minimal in the nesting ordering and $k_i=0$ otherwise.
Then if for some $\ell_1,\ldots,\ell_r$ we have
$\dft_{f,C}(\ell_1,\ldots,\ell_r)<s$, then since $\dft_{f,C}$ is increasing in
all variables, there is some $i$ such that $\ell_i\le K$.
\end{proof}

With this, we can now finally describe truncating a low-defect pair to a
specified defect:

\begin{thm}
\label{thmtrunc}
Let $(f,C)$ be a low-defect pair, say of degree $r$, let $s\ge 0$ be a real
number, and let $S=\{ (k_1,\ldots,k_r) : \dft_{f,C}(k_1,\ldots,k_r) < s\}$.
Then there exists a finite set $T\subseteq (\mathbb{Z}_{\ge0} \cup \{*\})^r$
such that:
\begin{enumerate}
\item We have $S = \bigcup_{p\in T} S(p)$.
\item For each $p$ in $T$, the set of $i$ for which $k_i\ne *$ corresponds to a
subset of the variables of $f$ which is downward closed (under the nesting
ordering); hence if $(g,D)$ denotes the $3$-substitution of $p$ into $(f,C)$,
then $(g,D)$ is a truncation of $(f,C)$.  Furthermore, we have $\dft(g,D)\le s$,
and hence $\deg g\le \lfloor s\rfloor$; and if $g$ has degree $0$, the former
inequality is strict.
\end{enumerate}
\end{thm}

\begin{proof}
We prove the statement by induction on $r$.

Suppose $r=0$, that is to say, $f$ is a constant $n$.  If $s>\dft(f,C)$, then
we may take $T=\{()\}$, where here $()$ indicates the unique element of
$(\mathbb{Z}_{\ge0} \cup \{*\})^0$.  For $S()=\{()\}$, and $S=\{()\}$ as well,
for $\dft_{f,C}()=C-3\log_3 n = \dft(f,C)<s$.  So the first condition is
satisfied.  For the second condition, the set of indices used is the empty set,
we have $(g,D)=(f,C)$ (hence $(g,D)$ is trivially a truncation), and so
$\dft(g,D)=\dft(f,C)<s$.

Otherwise, if $s\le \dft(f,C)$, we take $T=\emptyset$, so $\bigcup_{p\in T}
S(p)=\emptyset$.  Since, as was noted above, $\dft_{f,C}()=\dft(f,C)$, we have
$\dft_{f,C}()\ge s$, and hence $S=\emptyset$; thus the first condition is
satisfied.  The second condition is satisfied trivially.

Now suppose that $r>0$.  Once again, we have two cases.  If $s\ge \dft(f,C)$,
then we may take $T=\{(*,\ldots,*)\}$.  By Proposition~\ref{dftbd}, for any
$(k_1,\ldots,k_r)\in \mathbb{Z}_{\ge0}^r$, we have
$\dft_{f,C}(k_1,\ldots,k_r)<\dft(f,C)\le s$, i.e.
$S=\mathbb{Z}_{\ge0}^r=S(*,\ldots,*)$, satisfying the first condition.  For the
second condition, we once again have that the set of indices used is the null
set, so $(g,D)=(f,C)$, and so is trivially a truncation, and
$\dft(g,D)=\dft(f,C)\le s$.

This leaves the case where $r>0$ and $s<\dft(f,C)$.  In this case, we may apply
Corollary~\ref{dirtrunc2}, and choose a $K$ such that whenever
$\dft_{f,C}(k_1,\ldots,k_r)<s$, then $k_i \le K$ for some $i$ which is minimal
in the nesting ordering.  That is to say, if we define
\[ T_0 := \{(*,\ldots,*,k_i,*,\ldots,*) : x_i\ \textrm{minimal in nesting
ordering},\ k_i\le K\}, \]
then $S\subseteq \bigcup_{p\in T_0} S(p)$, and for each $p\in T$, the
$3$-substitution of $p$ into $(f,C)$ is a direct truncation of $(f,C)$.
However, we still do not necessarily have that $\dft(g,D)\le s$, nor do we
necessarily have equality in the first condition.  This is where we apply the
inductive hypothesis.

For each $p\in T_0$, let $(g_p, D_p)$ be the $3$-substitution of $p$ into
$(f,C)$; this is a direct truncation of $(f,C)$.  Apply the inductive hypothesis
to each $(g_p,D_p)$ to obtain $T_p \subseteq
(\mathbb{Z}_{\ge0}\cup\{*\})^{r-1}$.  We can then pull this back to $T'_p
\subseteq (\mathbb{Z}_{\ge0}\cup\{*\})^r$; since $p=(*,\ldots,*,k_i,*,\ldots)$
for some position $i$ and some number $k_i$, we can pull back
$q=(\ell_1,\ldots,\ell_{i-1},\ell_{i+1},\ldots,\ell_r)\in T_p$ (where here we
may have $\ell_j= *$) to
$q':=(\ell_1,\ldots,\ell_{i-1},k_i,\ell_{i+1},\ldots,\ell_r)$.  Finally we can
take $T=\bigcup_{p\in T_0} T'_p$.

It remains to show that $T$ has the desired properties.  Say we have an element
of $T$; it is an element of some $T'_p$, i.e., with the notation above, it has
the form $q'$ for some $q\in T_p$.  Say $p=(*,\ldots,*,k_i,*,\ldots)$.  The
indices used in $q$ correspond to some downward closed subset of the variables
of $(g_p, D_p)$, i.e.~to a downward closed subset of
$\{x_1,\ldots,x_{i-1},x_{i+1},\ldots,x_r\}$.  Since $x_i$ is minimal in
$\{x_1,\ldots,x_r\}$, adding it in again results in a downward closed set.

Now we check that $S\subseteq \bigcup_{p\in T} S(p)$.  Say
$\dft_{f,C}(k_1,\ldots,k_r)<s$; then there is some $i$ with $x_i$ minimal and
$k_i\le K$.  Let $p$ be the corresponding element of $T_0$ and $(g_p,D_p)$ as
above.  Then by Proposition~\ref{dirtrunc},
$\dft_{g_p,D_p}(k_1,\ldots,k_{i-1},k_{i+1},\ldots,k_r) =
\dft_{f,C}(k_1,\ldots,k_r)$, and so $(k_1,\ldots,k_{i-1},k_{i+1},\ldots,k_r)\in
T_p$, and so $(k_1,\ldots,k_r)\in T'_p\subseteq T$, as needed.

Suppose now that we take an element of $T$; write it as $q'\in T'_p$ for some
$p$ and some $q\in T_p$, using the notation above.  Then $(g_{q'},D_{q'})$ can
also be obtained by $3$-susbtituting $q$ into $(g_p,D_p)$; hence by the
inductive hypothesis, $\dft(g_{q'},D_{q'})\le s$, and this is strict if $\deg
g_{q'}=0$.  This then proves as well that $S\supseteq \bigcup_{p\in T}S(p)$; say
$q'\in T$, write $q'=(k_1,\ldots,k_r)$, and let $i_1,\ldots,i_s$ be the indices
for which $k_i= *$.  Then for $(\ell_1,\ldots,\ell_r)\in S(q')$, we may write
$\dft_{f,C}(\ell_1,\ldots,\ell_r) =
\dft_{g_{q'},D_{q'}}(\ell_{i_1},\ldots,\ell_{i_s})$, and this latter is less
than $s$, since it is at most $\dft(g_{q'},D_{q'})$, and strictly less than it
if $\deg g_{q'}>0$.  This proves the theorem.
\end{proof}

And if we can truncate one low-defect polynomial to a given defect, we can
truncate many low-defect polynomials to that same defect.  Here, at last, is the
result of taking the ``building-up'' Theorem~\ref{oldmainthm}, and applying our
new ``filtering-down'' step:

\begin{thm}
\label{mainthm}
For any real $s\ge 0$, there exists a finite set $\sS_s$ of low-defect pairs
satisfying the following conditions:
\begin{enumerate}
\item For any $n\in B_s$, there is some low-defect pair in $\sS_s$ that
efficiently $3$-represents $n$.
\item Each pair $(f,C)\in \sS_s$ satisfies $\dft(f,C)\le s$, and hence $\deg
f\le \lfloor s\rfloor$; and if $f$ has degree $0$, the former inequality is
strict.
\end{enumerate}
\end{thm}

\begin{proof}
By Theorem~\ref{oldmainthm}, there exists a finite set $\sT_s$ of low-defect
pairs such that for any $n \in B_s$, there is some low-defect pair in $\sT_s$
that efficiently $3$-represents $s$.  (Indeed, by Theorem~\ref{oldmainthm}, we
may even choose $\sT_s$ to only consist of polynomials of degree at most
$\lfloor s \rfloor$, but this is not needed.)

Now for each $(f,C)\in \sT_s$, take $T_{f,C}$ as provided by
Theorem~\ref{thmtrunc}; define $\sT_{f,C}$ to be the set
\[\{ (g,D) : (g,D)\ \textrm{is a $3$-substitution of $p$ into $(f,C)$},\ p\in
\sT_{f,C}\};\]
this is a set of low-defect pairs by condition (2) of Theorem~\ref{thmtrunc}.
We can then define $\sS_s$ to be the union of the $\sT_{f,C}$.  We see
immediately that $\sS$ satisfies condition (2) of the theorem, as this follows
from condition (2) of Theorem~\ref{thmtrunc}.

To verify condition (1), say $n\in B_s$.  Then there is some $(f,C)\in
\sT_s$ that efficiently $3$-represents $n$; say
$n=f(3^{\ell_1},\ldots,3^{\ell_r})$ with $\cpx{n}=C+3(\ell_1+\ldots+\ell_r)$, so
$\dft_{f,C}(\ell_1,\ldots,\ell_r)=\dft(n)<s$.  Then $(\ell_1,\ldots,\ell_r)\in
S(p)$ for some $p\in T_{f,C}$.  Say $p=(k_1,\ldots,k_r)$, and let $i_1,\ldots,
i_s$ be the indices for which $k_i=*$.  Then if we let $(g,D)$ be the
$3$-substitution of $p$ into $(f,C)$, then
$n=f(3^{\ell_1},\ldots,3^{\ell_r})=g(3^{\ell_{i_1}},\ldots,3^{\ell_{i_s}})$, and
$\cpx{n}=C+3(\ell_1+\ldots+\ell_r)=D+3(\ell_{i_1}+\ldots+\ell_{i_s})$, so $n$ is
efficiently $3$-represented by $(g,D)\in \sT_{f,C}\subseteq \sS_s$.
\end{proof}

Note that although such a covering of $B_r$ cannot produce extraneous numbers in
the sense of $3$-representing numbers whose defects are too high, it can still
$3$-represent numbers that are not leaders.

We then obtain Theorem~\ref{mainthmfront} as a corollary:

\begin{proof}[Proof of Theorem~\ref{mainthmfront}]
Given $s$, we may consider a set $\sS_s$ of low-defect pairs as described in
Theorem~\ref{mainthm}.  We may then define $\sT_s$ to be the set of low-defect
polynomials used in these pairs.  Then if $\dft(n)<s$, $n$ is $3$-represented by
$\xpdd{f}$ for some $f\in \sS_s$.  Conversely, if $n$ is $3$-represented by
$\xpdd{f}$ for some $f\in \sS_s$, then either $\deg f>0$, in which case
$\dft(n)<\dft(f)\le s$, or $\deg f=0$, in which case $\dft(n)\le\dft(f)<s$.
\end{proof}

\subsection*{Acknowledgements}
The author is grateful to J.~Arias de Reyna and E.~H.~Brooks for helpful
discussion. He thanks his advisor J.~C.~Lagarias for help with editing and
further discussion.
Work of the author was supported by NSF grants DMS-0943832 and DMS-1101373.

\appendix
\section{Representing closed intervals}
\label{nonstrict}

It's worth noting that the theorems above about $A_r$ and $B_r$, and how to
build up coverings for them, etc., are formulated in terms of $A_r$ and $B_r$,
which are defined by the strict inequality $\dft(n)<r$.  In many contexts,
however, it is more natural to consider the nonstrict inequality $\dft(n)\le r$.
So let us define:

\begin{defn}
For a real number $r\ge 0$, the set $\overline{A}_r$ is the set $\{ n\in\N:
\dft(n)\le r\}$.  The set $\overline{B}_r$ is the set of all elements of
$\overline{A}_r$ which are leaders.
\end{defn}

We can then also define:
\begin{defn}
A finite set $\sS$ of low-defect pairs will be called a \emph{covering set} for
$\overline{B}_r$ if, for every $n\in \overline{B}_r$, there is some low-defect
pair in $\sS$ that efficiently $3$-represents it.
\end{defn}

One can then write down theorems about $\overline{A}_r$ and $\overline{B}_r$
similar to those above and in \cite{paper1} and \cite{paperwo} about $A_r$ and
$B_r$.  We will state them here without proof, as the proofs are the same except
for the strictnesses of some of the inequalities.

\begin{thm}
For any real $0\le \alpha<1$, $\overline{B}_\alpha$ is a finite set.
\end{thm}

\begin{thm}
Suppose that $0< \alpha <1$ and that $k\ge1$.
Then any $n\in \overline{B}_{(k+1)\alpha}$ can be most-efficiently
represented in (at least) one of the following forms:
\begin{enumerate}
\item
For $k=1$,
there is either a good factorization $n=u\cdot v$ where
$u,v\in \overline{B}_\alpha$, or a good factorization $n=u\cdot v\cdot w$ with
$u,v,w\in \overline{B}_\alpha$; \\
For $k \ge 2$, there is a good factorization $n=u \cdot v$ where $u\in
\overline{B}_{i\alpha}$,
$v\in \overline{B}_{j\alpha}$ with $i+j=k+2$ and $2\le i, j\le k$.
\item $n=a+b$ with $\cpx{n}=\cpx{a}+\cpx{b}$, $a\in \overline{A}_{k\alpha}$,
$b\le a$ a solid number and
\[\dft(a)+\cpx{b}\le(k+1)\alpha+3\log_3 2.\]
\item There is a good factorization $n=(a+b)v$ with $v\in \overline{B}_\alpha$,
$a+b$ being a most-efficient representation, and $a$ and $b$ satisfying the
conditions in the case (2) above.
\item $n\in T_\alpha$, where $T_\alpha$ is as defined in \cite{paper1} (and thus
in particular either $n=1$ or $\cpx{n}=\cpx{n-1}+1$.)
\item There is a good factorization $n = u\cdot v$ with $u\in T_\alpha$ and
$v\in \overline{B}_\alpha$.
\end{enumerate}
\end{thm}

(Note here that we can use the same $T_\alpha$ from \cite{paper1} with no
alterations.)

\begin{thm}
For any real $r\ge 0$, there exists a finite covering set $\sS_r$ for
$\overline{B}_r$.  Furthermore, we can choose $\sS_r$ such that each $(f,C)\in
\sS_r$ has degree at most $\lfloor r \rfloor$.
\end{thm}

\begin{thm}
Let $(f,C)$ be a low-defect pair, say of degree $r$, let $s\ge 0$ be a real
number, and let $S=\{ (k_1,\ldots,k_r) : \dft_{f,C}(k_1,\ldots,k_r) \le s\}$.
Then there exists a finite set $T\subseteq (\mathbb{Z}_{\ge0} \cup \{*\})^r$
such that:
\begin{enumerate}
\item We have $S = \bigcup_{p\in T} S(p)$.
\item For each $p$ in $T$, the set of $i$ for which $k_i\ne *$ corresponds to a
subset of the variables of $f$ which is downward closed (under the nesting
ordering); hence if $(g,D)$ denotes the $3$-substitution of $p$ into $(f,C)$,
then $(g,D)$ is a truncation of $(f,C)$.  Furthermore, we have $\dft(g,D)\le s$,
and hence $\deg g\le \lfloor s\rfloor$.
\end{enumerate}
\end{thm}

\begin{thm}
For any real $s\ge 0$, there exists a finite set $\sS_s$ of low-defect pairs
satisfying the following conditions:
\begin{enumerate}
\item For any $n\in \overline{B}_s$, there is some low-defect pair in $\sS_s$
that efficiently $3$-represents $n$.
\item Each pair $(f,C)\in \sS_s$ satisfies $\dft(f,C)\le s$, and hence $\deg
f\le \lfloor s\rfloor$.
\end{enumerate}
\end{thm}

\end{document}